\documentclass[a4paper,reqno]{amsart}

\usepackage[utf8]{inputenc}
\usepackage{amssymb}
\usepackage{amsmath}
\usepackage{hyperref}
\usepackage{tikz-cd}
\usepackage{mathtools}
\usepackage[all]{xy}
\usepackage{enumitem}
\usepackage{aliascnt} 

\hypersetup{
    colorlinks,
    linkcolor={red!50!black},
    citecolor={blue!50!black},
    urlcolor={blue!80!black}
}

\newtheorem{theorem}{Theorem}[section]

 \newaliascnt{lemma}{theorem}
  \newtheorem{lemma}[lemma]{Lemma}
  \aliascntresetthe{lemma}

  \newaliascnt{proposition}{theorem}
  \newtheorem{proposition}[proposition]{Proposition}
  \aliascntresetthe{proposition}

  \newaliascnt{corollary}{theorem}
  
  \aliascntresetthe{corollary}

 \newaliascnt{condition}{theorem}
\newtheorem{condition}[condition]{Condition}
\aliascntresetthe{condition}

\theoremstyle{definition}

 \newaliascnt{definition}{theorem}
\newtheorem{definition}[definition]{Definition}
\aliascntresetthe{definition}

  \newaliascnt{remark}{theorem}
  \newtheorem{remark}[remark]{Remark}
  \aliascntresetthe{remark}

\renewcommand{\_}{\underline{\,\,\,\,}}

\newcommand{\CC}{\ensuremath{\mathbb{C}}}
\newcommand{\EE}{\ensuremath{\mathbb{E}}}

\newcommand{\PP}{\ensuremath{\mathbb{P}}}

\newcommand{\ZZ}{\ensuremath{\mathbb{Z}}} 
\newcommand{\WW}{\ensuremath{\mathbb{W}}}

\newcommand{\cC}{\mathcal{C}}

\newcommand{\cE}{\mathcal{E}}
\newcommand{\cF}{\mathcal{F}}
\newcommand{\cG}{\mathcal{G}}

\newcommand{\cM}{\mathcal{M}}

\newcommand{\cO}{\mathcal{O}}

\newcommand{\cT}{\mathcal{T}}
\newcommand{\cS}{\mathcal{S}}

\newcommand{\cR}{\mathcal{R}}
\newcommand{\reg}{\mathcal{O}}

\newcommand{\sym}{\mathfrak{S}}

\newcommand{\fa}{\mathfrak{a}}

\newcommand{\fS}{\mathfrak{S}}

\DeclareMathOperator{\Coh}{Coh}

\DeclareMathOperator{\End}{End} 
\DeclareMathOperator{\Ext}{Ext}

\DeclareMathOperator{\Hom}{Hom}

\DeclareMathOperator{\id}{id}

\DeclareMathOperator{\NS}{NS}

\DeclareMathOperator{\Pic}{Pic}

\DeclareMathOperator{\rk}{rk}

\DeclareMathOperator{\Ind}{Ind}
\DeclareMathOperator{\Res}{Res}

\DeclareMathOperator\Db{D^b}
\DeclareMathOperator\DbS{D^b_{\mathfrak S_n}}

\DeclareMathOperator{\irr}{irr}

\title[Moduli spaces of generalised tautological bundles]{Moduli spaces of generalised tautological bundles on Hilbert schemes}

\author{Andreas Krug}
\address{Institut f\"ur Algebraische Geometrie, Leibniz Universit\"at Hannover, Welfengarten 1, 30167 Hannover, Germany}
\email{krug@math.uni-hannover.de}

\author{Fabian Reede}
\address{Freie Waldorfschule Hannover-Maschsee, Germany}
\email{reede@waldorfschule-maschsee.de}

\author{Ziyu Zhang}
\address{Institute of Mathematical Sciences, ShanghaiTech University, 393 Middle Huaxia Road, 201210 Shanghai, P.R.China}
\email{zhangziyu@shanghaitech.edu.cn}


\begin{document}

\begin{abstract}
We construct new stable vector bundles on Hilbert schemes of points on algebraic surfaces, which are parametrised by connected components of their moduli spaces. This work generalises aspects of our previous work on tautological bundles and of recent work of O'Grady.
\end{abstract}

\maketitle

\setcounter{section}{-1}

\section{Introduction}

\subsection{Background}

Let $X$ be a smooth projective variety over the field of complex numbers. The moduli space $M_X$ of stable sheaves on $X$ reflects much geometric information of $X$. These moduli spaces are thoroughly studied when $\dim X \leqslant 2$. However, a lot remains unknown when $X$ is of higher dimension.

Let $S$ be a smooth projective surface, and $X = S^{[n]}$ be the Hilbert scheme of $0$-dimensional subschemes of length $n$ on $S$. The moduli space of stable sheaves $M_{S^{[n]}}$ in this particular case has attracted wide attention; see \cite{schlick_10, wandel_tautological_2016, stapleton_taut_2016, reede_zhang_22, ogrady_rigid} among others. Since the general behavior of these moduli spaces tends to be complicated, much effort was devoted to construct particular examples of stable sheaves on $S^{[n]}$.

Among the examples that have been constructed so far, there is a particular class called the \emph{tautological bundles}; see e.g.\ \cite{stapleton_taut_2016}. Under favorable conditions, the tautological bundles are stable, and their parameter spaces become components of $M_{S^{[n]}}$. In the present manuscript, we aim to generalise the notion of tautological bundles and construct a new series of components of $M_{S^{[n]}}$.

\subsection{Construction}

We study the following type of bundles on $S^{[n]}$, which will be called \emph{generalised tautological bundles}. More precisely, let $$\lambda=(\lambda_1, \lambda_2, \dots, \lambda_k)\dashv n$$ be a partition of $n$, and
\[
\sym_\lambda \coloneqq \sym_{\lambda_1}\times \sym_{\lambda_2}\times \dots \times \sym_{\lambda_k}\leqslant \sym_n
\]
the stabliliser of the associated partition of the set $\{1,2,\dots , n\}$.
For each $j=1,\dots, k$, let $E_j$ be a vector bundle on $S$ and let $W_j$ an irreducible representation of $\sym_{\lambda_j}$. We consider
$$ G_\lambda^\WW(E_1, \dots, E_k) \coloneqq \Ind_{\sym_{\lambda}}^{\sym_n}\Bigl((E_1^{\boxtimes \lambda_1}\otimes W_1)\boxtimes (E_2^{\boxtimes \lambda_2}\otimes W_2)\boxtimes \dots \boxtimes (E_k^{\boxtimes \lambda_k}\otimes W_k) \Bigr), $$
which is an $\sym_n$-equivariant vector bundle on $S^n$. Via the derived equivalence
$$ \Psi \colon \DbS(S^n) \longrightarrow \Db(S^{[n]}) $$
established by Bridgeland-King-Reid and Haiman in \cite{BKR, Hai}, we obtain
$$ F_\lambda^\WW(E_1, \dots, E_k) \coloneqq \Psi \left( G_\lambda^\WW(E_1, \dots, E_k) \right). $$
We will see in \autoref{subsect:gentaut} that $F_\lambda^\WW(E_1, \dots, E_k)$ is a vector bundle on $S^{[n]}$, and call it the \emph{generalised tautological bundle} associated to $E_1, \dots, E_k$ and $W_1, \dots, W_k$.

\subsection{Main results}

The following theorems are the main results of the present manuscript. Similar to the classical situation, the generalised tautological bundle $F_\lambda^\WW(E_1, \dots, E_k)$ inherits stability from its ingredients $E_1, \dots, E_k$.

\begin{theorem}[\autoref{same_div}]\label{thm:main1}
	Assume that $E_1, \dots, E_k$ are pairwise non-isomorphic slope stable vector bundles with respect to an ample line bundle $H$ on $S$. Then $F_\lambda^\WW(E_1, \dots, E_k)$ is a slope stable vector bundle with respect to a suitable ample line bundle $\tilde{H}$ on $S^{[n]}$.
\end{theorem}

Moreover, when the components $E_1, \dots, E_k$ deform in their own moduli spaces, the generalised tautological bundle deform along with them in families.

\begin{theorem}[\autoref{thm:moduli-main-thm}]\label{thm:main2}
	Let $M_1, \dots, M_k$ be connected components of moduli spaces of slope stable vector bundles with respect to an ample line bundle $H$ on $S$. We assume that $M_1, \dots, M_k$ are smooth and projective, and that the Young diagrams associated to $W_1, \dots, W_k$ are rectangular. Then possibly after tensoring all bundles parametrised by $M_j$ with a line bundle $L_j$ for $j=1, \dots, k$ (see \autoref{cond:moduli_dist} and \autoref{cond:Homvanish}), there exists a morphism
	\begin{equation*}
		\varphi_\lambda^\WW \colon M_1 \times \dots \times M_k \longrightarrow M_{S^{[n]}}\quad,\quad \bigl([E_1],\dots,[E_k] \bigr)\mapsto[F_\lambda^\WW(E_1, \dots, E_k)]
	\end{equation*}
	which is an isomorphism to a connected component of $M_{S^{[n]}}$.
\end{theorem}

This theorem identifies many smooth projective connected components of the moduli space $M_{S^{[n]}}$, which are mostly of higher dimensions.

Another result we prove in the present manuscript is a closed formula for the first Chern classes of generalised tautological bundles; see \autoref{thm:c1}. This is logically independent of the above theorems. However, besides general curiosity, there are two other important reasons for carrying out this computation explicitly.

First of all, for the stability of the generalised tautological bundle $F_\lambda^\WW(E_1, \dots, E_k)$, it is necessary to require that the stable bundles $E_i$ are \emph{pairwise non-isomorphic}; see \autoref{rem:counterexamplestab} and \autoref{rmk:another-counter}. Without this technical assumption, $F_\lambda^\WW(E_1, \dots, E_k)$ could become unstable, which means that 
$\varphi_\lambda^\WW$ is not well-defined over the diagonals of $M_1\times \dots\times M_k$ if some of the moduli spaces $M_i$ coincide. 
In this case, we expect some blow-up of $M_1\times\dots\times M_k$, or a finite group quotient thereof, to be isomorphic to a connected component of $M_{S^{[n]}}$; compare \cite{krug_smooth_2024} for this phenomenon in a similar situation. 
 We plan to come back to this issue in a subsequence work, where the formula for the first Chern class will be used to find a destabilising subbundle.

Secondly, when $S$ is a K3 surface, the Hilbert scheme $S^{[n]}$ gives an important series of examples of compact hyperkähler manifolds. In such a case there is a class of torsion free sheaves on $S^{[n]}$ of particular interests, called the \emph{modular sheaves}, whose behaviors are similar to those on K3 surfaces; see \cite{ogrady_many, ogrady_rigid} among others. Since modular sheaves are characterised by a topological condition involving the first two Chern classes, and our strategy can likely be helpful for computing the second Chern class as well, this formula could serve as the starting point to study when the generalized tautological bundles are modular. Note that many of the examples of modular sheaves provided in \cite{ogrady_many, ogrady_rigid} belong to the class of generalized tautological bundles.

\subsection{Structure of the paper}

The main content of this manuscript consists of 4 sections. In \autoref{sect:gentaut} we construct the class of generalised tautological bundles and study their first properties. Then we prove \autoref{thm:main1} in \autoref{sect:stab}, which establishes the stability of the generalised tautological bundles under appropriate polarisation. In \autoref{sect:moduli} we study moduli spaces of generalised tautological bundles and prove \autoref{thm:main2}, which characterise certain components of stable bundles on the Hilbert scheme. Finally we compute the first Chern classes of such bundles in \autoref{sect:Chern}.

\subsection{Acknowledgements}

Z.~Zhang is supported by National Natural Science Foundation of China (Grant No.~12371046) and Science and Technology Commission of Shanghai Municipality (Grant No.~22JC1402700).

\section{Generalised tautological bundles}\label{sect:gentaut}

The main objects that we study in this paper are generalizations of tautological bundles on Hilbert schemes. In this section, we give an explicit construction of these objects and study some basic properties.

\subsection{The Bridgeland-King-Reid-Haiman equivalence}

Let $S$ be a smooth projective surface, then we have the product $S^n$ of $n$ copies of $S$, the symmetric product $S^{(n)} \coloneqq S^n/\mathfrak{S}_n$, and the Hilbert scheme $S^{[n]}$ of $n$ points on $S$. The isospectral Hilbert scheme is defined by $I^nS \coloneqq (S^{[n]}\times_{S^{(n)}}S^n)_{\mathrm{red}}$, which fits in a commutative diagram
\begin{equation}\label{eqn:maindiagram}
	\begin{tikzcd}
		I^nS \arrow[r,"p"]\arrow[d,swap,"q"] & S^n\arrow[d,"\pi"]\\
		S^{[n]}\arrow[r,"\mu"']& S^{(n)}
	\end{tikzcd}
\end{equation}
in which the top row is equipped with an $\mathfrak{S}_n$-action. We recall the derived McKay correspondence of Bridgeland-King-Reid and Haiman \cite{BKR, Hai} in the version of \cite{krug_remarks_2018}:
\begin{equation}\label{eqn:derived_McKay}
	\Psi \colon \DbS(S^n)\longrightarrow \Db(S^{[n]}), \quad E \longmapsto q_\ast^{\mathfrak{S}_n} \circ Lp^\ast (E),
\end{equation}
which is an equivalence of derived categories. Note that $$q_{*}^{\mathfrak{S}_n}(\_) \coloneqq (\_)^{\sym_n}\circ q_*$$ denotes the composition the push-forward and the functor of taking invariants. It does not need to be derived as $q$ is finite, and taking invariants is exact in characteristic 0.

For any variety $T$, we can also consider the relative version
\begin{equation*}
	\begin{tikzcd}
		I^nS \times T \arrow[r,"p'"]\arrow[d,swap,"q'"] & S^n \times T \arrow[d,"\pi'"]\\
		S^{[n]} \times T \arrow[r,"\mu'"']& S^{(n)} \times T
	\end{tikzcd}
\end{equation*}
where each arrow represents the product of the corresponding morphism in \eqref{eqn:maindiagram} and $\id_T$. The top row is still equipped with an $\fS_n$-action on its first factor, and we can similarly define
\begin{equation*}
	\Psi' \coloneqq {q'_\ast}^{\fS_n} \circ Lp'^\ast. 
\end{equation*}

The following result shows that the above functors preserve local freeness.

\begin{lemma}\label{BKR-fam}
Let $\mathcal{E}$ be an $\mathfrak{S}_n$-equivariant locally free sheaf on $S^n\times T$. Then $\Psi'(\mathcal{E})$ is a locally free sheaf satisfying $\Psi'(\mathcal{E})_t\cong\Psi(\mathcal{E}_t)$ for every $t\in T$.
\end{lemma}

\begin{proof}
We first show that $\Psi'(\mathcal{E})$ is locally free on $S^{[n]}\times T$.  Since $\mathcal{E}$ is locally free, $Lp'^\ast \mathcal{E}\cong p'^\ast\mathcal{E}$ is also locally free on $I^nS\times T$. Since $q$ is finite and flat, so is $q'$. Hence $q'_{*}p'^{*}\mathcal{E}$ is locally free on $S^{[n]}\times T$. Since the $\mathfrak{S}_n$-invariants form a direct summand, we find that $\Psi'(\mathcal{E})$ is locally free on $S^{[n]}\times T$.

To see the isomorphism $\Psi'(\mathcal{E})_t\cong\Psi(\mathcal{E}_t)$ we consider the Cartesian diagram
\begin{equation*}
	\begin{tikzcd}
		S^{[n]}\times T & I^nS\times T\arrow[l,swap,"q'"] \arrow[r,"p'"] & S^n\times T\\
		S^{[n]}\arrow[u,hook',"\iota_{t,{S^{[n]}}}"]&  I^nS\arrow[l,swap,"q"]\arrow[r,"p"]\arrow[u,hook',"\iota_{t,I^nS}"] & S^n\arrow[u,hook',"\iota_{t,{S^n}}"]
	\end{tikzcd}
\end{equation*}
where the vertical maps are inclusions of the fibers over $t\in T$. The usual base change gives
\begin{equation*}
	\iota_{t,S^{[n]}}^\ast q'_\ast p'^\ast \mathcal{E}\cong q_\ast p^\ast \iota_{t,S^n}^\ast \mathcal{E}\cong q_\ast p^\ast \mathcal{E}_t.
\end{equation*}
Taking $\mathfrak{S}_n$-invariants then shows that $\Psi'(\mathcal{E})_t\cong \Psi(\mathcal{E}_t)$.
\end{proof}

\subsection{Generalised tautological bundles}\label{subsect:gentaut}

We apply the Bridgeland-King-Reid-Haiman equivalence to define the main objects of this paper. Given a partition $$\lambda=(\lambda_1,\dots,\lambda_k) \dashv n$$ of $n \in \mathbb{N}$ and locally free sheaves $E_1,\dots,E_k$, we consider the locally free sheaf
\begin{equation*}
	\mathbb{E} \coloneqq E_1^{\boxtimes\lambda_1}\boxtimes\dots\boxtimes E_k^{\boxtimes \lambda_k} \in \Coh(S^n).
\end{equation*}
Note that $\mathbb{E}$ is naturally a $\mathfrak{S}_\lambda$-equivariant locally free sheaf, where
\begin{equation*}
\mathfrak{S}_{\lambda} \coloneqq \mathfrak{S}_{\lambda_1}\times \dots \times\mathfrak{S}_{\lambda_k}\leqslant \mathfrak{S}_n.
\end{equation*}
Moreover we choose an irreducible $\mathfrak{S}_{\lambda_j}$-representation $W_j$ for $j=1,\dots,k$. Then
\begin{equation*}
	\mathbb{W} \coloneqq W_1 \otimes\dots\otimes W_k
\end{equation*}
is an irreducible $\mathfrak{S}_{\lambda}$-representation by \cite[p.27, Theorem 10]{serre_linear_1977}. It follows that $\mathbb{E}\otimes\mathbb{W}$ is also an $\mathfrak{S}_{\lambda}$-equivariant vector bundle. We turn it into an $\mathfrak{S}_n$-equivariant vector bundle by applying the induction functor (the both-sided adjoint of the restriction functor; see e.g.\ \cite[Sect.\ 3.2]{BOber--equi} for details)
\begin{equation}\label{eqn:Gdef}
		G_\lambda^{\mathbb{W}}(E_1,\ldots,E_k) \coloneqq 
    \Ind_{\mathfrak{S}_\lambda}^{\mathfrak{S}_n} \left( \mathbb{E}\otimes\mathbb{W}  \right) \in \Coh_{\sym_n}(S^n).
\end{equation}
Via the derived McKay correspondence \eqref{eqn:derived_McKay}, we obtain
\begin{equation}\label{eqn:Fdef}
		F_\lambda^{\mathbb{W}}(E_1,\ldots,E_k) \coloneqq \Psi\left( G_\lambda^{\mathbb{W}}(E_1,\ldots,E_k)\right) \in \Db(S^{[n]}).
\end{equation}
It follows immediately from \autoref{BKR-fam} (with $T$ being a point) that $F_\lambda^{\mathbb{W}}(E_1,\ldots,E_k)$ is a locally free sheaf.

\begin{definition}
	The locally free sheaf $F_\lambda^{\mathbb{W}}(E_1,\ldots,E_k)$ is called the \emph{generalised tautological bundle} associated to the partition $\lambda = (\lambda_1, \dots, \lambda_k)$ of $n$, the locally free sheaves $E_1, \dots, E_k$, and the representations $W_1, \dots, W_k$.
\end{definition}

\begin{remark}
	To justify the name, we point out that the classical notion of tautological bundles is a special case of the above construction. Indeed, when $\lambda = (n-1, 1)$, $E_1 = \cO_S$, $E_2 = E$, $W_1 = \mathbf{1}$ and $W_2 = \mathbf{1}$ (where $\mathbf{1}$ is the trivial representation), we obtain an isomorphism to the classical tautological bundle
	\begin{equation}\label{eq:getbacktaut}
		F_{\lambda}^\WW(\cO_S, E) \cong E^{[n]},
	\end{equation}
	which was proven in \cite[Theorem 3.6]{krug_remarks_2018} (note that, in the notation of \emph{loc.\ cit.}, $G_{(n-1,1)}^{\mathbf 1}(\reg_S,E)= \mathsf C(E)$).
\end{remark}

The above construction can also be carried out in families. For each $j = 1, \dots k$, we assume that $\cE_j$ is a locally free sheaf on $S \times T_j$, viewed as a family of locally free sheaves on $S$ parametrised by $T_j$. We denote
$$ \mathcal{T} \coloneqq T_1 \times \dots \times T_k. $$
For $i = 1, \dots, n$ and $j = 1, \dots, k$, we consider the double projection
$$ \pi_{i,j} \colon \quad S^n \times \mathcal{T} \longrightarrow S \times T_j. $$
The relative version of \eqref{eqn:Gdef} is given as
\begin{equation}\label{eqn:family-G}
	\mathcal{G}_\lambda^{\mathbb{W}}(\mathcal{E}_1,\ldots,\mathcal{E}_k) \coloneqq \Ind_{\mathfrak{S}_\lambda}^{\mathfrak{S}_n} \left( \bigotimes_{j=1}^k \left( \bigotimes_{i=\lambda_1 + \dots + \lambda_{j-1}+1}^{\lambda_1 + \dots + \lambda_j} \pi_{i,j}^\ast\mathcal{E}_j \right) \otimes \mathbb{W} \right),
\end{equation}
which can be viewed as a family of $\mathfrak{S}_n$-equivariant locally free sheaves on $S^n$ parametrized by $\mathcal{T}$, whose fiber over an arbitrary point $t = (t_1, \dots, t_k) \in \mathcal{T}$ is
\begin{equation*}
	\mathcal{G}_\lambda^{\mathbb{W}}(\mathcal{E}_1,\ldots,\mathcal{E}_k)_t\cong G_\lambda^{\mathbb{W}}(E_1,\ldots,E_k).
\end{equation*}
with $E_i\coloneqq \left(\mathcal{E}_i\right)_{t_i}$ the fiber of $\mathcal{E}_i$ over $t_i$. Moreover, we see by \autoref{BKR-fam} that
\begin{equation}\label{eqn:family-F}
	\mathcal{F}_\lambda^{\mathbb{W}}(\mathcal{E}_1,\ldots,\mathcal{E}_k) \coloneqq \Psi'(\mathcal{G}_\lambda^{\mathbb{W}}(\mathcal{E}_1,\ldots,\mathcal{E}_k))
\end{equation}
is a locally free sheaf on $S^{[n]} \times \cT$, with fiber over the point $t \in \cT$ given by
\begin{align}
	\mathcal{F}_\lambda^{\mathbb{W}}(\mathcal{E}_1,\ldots,\mathcal{E}_k)_t&=\Psi'(\mathcal{G}_\lambda^{\mathbb{W}}(\mathcal{E}_1,\ldots,\mathcal{E}_k))_t \notag\\
	&\cong \Psi(\mathcal{G}_\lambda^{\mathbb{W}}(\mathcal{E}_1,\ldots,\mathcal{E}_k)_t) \notag\\
	&\cong \Psi\left( G_\lambda^{\mathbb{W}}(E_1,\ldots,E_k)\right) \notag\\
	&=F_\lambda^{\mathbb{W}}(E_1,\ldots,E_k). \label{eqn:fiber-of-F}
\end{align}

In addition to the equivalence $\Psi$ in \eqref{eqn:derived_McKay}, there is another functor
\begin{equation}\label{eqn:stapleton}
	(\_)_{S^{n}} \colon \Coh(S^{[n]})\longrightarrow \Coh_{\mathfrak{S}_n}(S^n).
\end{equation}
defined in \cite[Section 1]{stapleton_taut_2016}, which will be called the \emph{Stapleton functor}. 
In contrast to $\Psi^{-1}$, it is not an equivalence. However, it has the advantage that it preserves subsheaves, and has some compatibility with slopes; see \cite[Lemma 1.2]{stapleton_taut_2016}. However, on generalized tautological bundles, it agrees with $\Psi^{-1}$ by the following 
\begin{lemma}\label{compare_stap}
	There exists an isomorphism 
	\begin{equation*}
		\left( F_\lambda^{\mathbb{W}}(E_1,\ldots,E_k) \right)_{S^n}\cong G_\lambda^{\mathbb{W}}(E_1,\ldots,E_k).
	\end{equation*}
\end{lemma}

\begin{proof}
	Since $\Psi$ is an equivalence, it follows from \eqref{eqn:Fdef} that
	\begin{equation*}
		\Psi^{-1} \left( F_\lambda^{\mathbb{W}}(E_1,\ldots,E_k) \right )= G_\lambda^{\mathbb{W}}(E_1,\ldots,E_k).
	\end{equation*}	
	Since $G_\lambda^{\mathbb{W}}(E_1,\ldots,E_k)$ is locally free, hence in particular reflexive, \cite[Lemma 3.12]{reede_zhang_22} shows that there is an isomorphism
	\begin{equation*}
		\left(F_\lambda^{\mathbb{W}}(E_1,\ldots,E_k) \right) _{S^n}\cong\Psi^{-1}\left(F_\lambda^{\mathbb{W}}(E_1,\ldots,E_k) \right) =G_\lambda^{\mathbb{W}}(E_1,\ldots,E_k).\qedhere
	\end{equation*}
\end{proof}

\section{Stability}\label{sect:stab}

In this section, we will show that the generalised tautological bundles associated to stable bundles on the surface $S$ are stable under a suitable polarization on the Hilbert scheme $S^{[n]}$. The bridge relating the stability on $S$ and $S^{[n]}$ is the equivariant stability on $S^n$.

\subsection{Equivariant stability of box products}

First of all, we introduce the notion of equivariant stability. For simplicity we restrict ourselves to slope stability.

\begin{definition}
	Assume that a finite group $G$ acts on a polarized variety $(X,H)$, which means that $G$ acts on $X$, and $H$ is $G$-invariant. A $G$-equivariant torsion free sheaf $E$ on $X$ is said to be $G$-equivariantly slope stable with respect to $H$ if
	\begin{equation*}
		\mu_H(F)<\mu_H(E)
	\end{equation*}
    for every $G$-equivariant subsheaf $F$ with $0<\rk(F)<\rk(E)$.
\end{definition}

Note that equivariant stability is a weaker notion than (the ordinary) stability since we only need to test the slopes of equivariant subsheaves instead of arbitrary ones. The following result, inspired by \cite[Lemma 1.3]{stapleton_taut_2016}, establishes the equivariant stability of an induced equivariant sheaf from the ordinary stability of the original sheaf.

\begin{lemma}\label{ind_stable}
Let $G$ be a finite group acting on a polarized variety $(X,H)$, $G'\leqslant G$ a subgroup, and $G' \backslash G$ the set of right cosets of $G'$ in $G$. Assume that $E$ is an $G'$-equivariant locally free sheaf, which is slope stable with respect to $H$, such that 
\begin{equation}\label{eq:Eassumption}
g^{*}E\not\cong E \quad\text{for every } [g]\in G'\backslash G \text{ with } [g] \neq [\id].
\end{equation} 
Then for every irreducible $G'$-representation $W$, the $G$-equivariant torsion free sheaf $F=\Ind_{G'}^G(E\otimes W)$ is $G$-equivariantly slope stable with respect to $H$.
\end{lemma}

\begin{proof}
Let $r \coloneqq \dim W$, $t \coloneqq [G:G']$, and let $g_1=\id, g_2, \dots, g_t$ be a set of representatives of the right cosets of $G'$ in $G$. Then, for the underlying (non-equivariant) sheaf, we have the direct sum decomposition
\[
F\cong E^{\oplus r}\oplus g_2^*E^{\oplus r} \oplus \dots \oplus g_t^*E^{\oplus r}.
\]
By the $G$-invariance of $H$, all the $g_j^*E$ are still stable. Hence, $F$ is polystable with $\mu_H(F)=\mu_H(E)$. Hence, it suffices to prove that the only $G$-equivariant subsheaf $0\neq U\subseteq F$ with $\mu_H(U)=\mu_H(F)$ is $U=F$. By \cite[Cor.\ 1.6.11]{huybrechts_moduli}, such $U$ is a direct summand of $F$. 
By the Krull--Schmidt property of $\Coh(X)$ (see \cite[Thm.\ 2]{Atiyah--KS}), it follows that 
\[
U\cong E^{\oplus s_1}\oplus g_2^*E^{\oplus s_2}\oplus\dots \oplus g_t^*E^{\oplus s_t}
\]
where $0\leqslant s_i\leqslant r$ for each $i= 1, \dots, t$. By the $G$-equivariance of $U$, together with the fact that $g_i^*E$'s are pairwise non-isomorphic by \eqref{eq:Eassumption}, we have $s_1=s_2=\dots=s_t$, denoted as $s$. Note that $s>0$ since $U\neq 0$. Again, by assumption \eqref{eq:Eassumption} together with the fact that all $g_i^*E$'s are stable with the same slope, we have
\begin{equation}\label{eq:Homvanish}
\Hom(E, g_i^*E)=0 \quad \text{for} \quad j=2,\dots, t.    
\end{equation}
As $G'$-equivariant sheaves, the composition of embeddings
$E^{\oplus s}\hookrightarrow U\hookrightarrow F$ factors through $E\otimes W\hookrightarrow F$, so we have an embedding 
$E^{\oplus s}\hookrightarrow E\otimes W$ of $G'$-equivariant sheaves. As $E$ is stable, hence simple, applying $\Hom(E,\_)$ to this embedding gives an embedding of some $s$-dimensional $G'$-subrepresentation of $W$. The irreducibility of $W$ implies that $s=\dim W=r$. Hence $U=F$.
\end{proof}

We apply the above criterion to the specific case of our interest. To fix notations, let $H$ be an ample line bundle on $S$. Recall that $H$ induces an ample line bundle
\begin{equation*}
	H_{S^n} \coloneqq \tau_1^\ast H + \dots + \tau_n^\ast H
\end{equation*}
on $S^n$, where $\tau_1, \dots, \tau_n$ are the projections from $S^n$ to its factors. This line bundle descends to an ample line bundle $H_{S^{(n)}}$ on $S^{(n)}$. By pulling it back along the Hilbert-Chow morphism $\mu$, we obtain a semi-ample line bundle
\begin{equation*}
H_{S^{[n]}} \coloneqq \mu^{*}H_{S^{(n)}}
\end{equation*}
on $S^{[n]}$. We recall that this construction of $H_{S^{[n]}}$ is the inclusion of the first direct summand in the decomposition
	\begin{equation*}
		\NS(S^{[n]})\cong \left( \NS(S)\right) _{S^{[n]}}\oplus \NS(\operatorname{Alb}(S))\oplus \mathbb{Z}\delta,
	\end{equation*}
where $\delta$ is the diagonal divisor in $S^{[n]}$; this follows from \cite{Fog--Pic}, see \cite[Section 2]{Gira} for details. Then we have

\begin{proposition}\label{prop:box-stability}
    Assume that $E_1, \dots, E_k$ are pairwise non-isomorphic slope stable vector bundles on $S$ with respect to $H$. Then $G_\lambda^\WW(E_1, \dots, E_k)$ is an $\sym_n$-equivariantly slope stable vector bundle on $S^n$ with respect to $H_{S^n}$.
\end{proposition}

\begin{proof}
    By \cite[Lemma 4.7]{ogrady_rigid}, we see that the $\fS_\lambda$-equivariant vector bundle $\mathbb{E}=E_1^{\boxtimes \lambda_1}\boxtimes \dots \boxtimes E_k^{\boxtimes \lambda_k}$ is slope stable with respect to $H_{S^n}$. To apply \autoref{ind_stable} and show that $G_\lambda^\WW(E_1, \dots, E_k)$ is $\sym_n$-equivariantly slope stable, it suffices to show that $\EE \not\cong g^\ast \EE$ for each right coset $[g] \in \sym_\lambda \backslash \sym_n$ with $[g] \neq [\id]$. 
    
    We define $F_1,\dots, F_n$ by 
    \[F_i \coloneqq E_j\quad \text{for}\quad \lambda_1+\dots +\lambda_{j-1}< i \leqslant \lambda_1+\dots +\lambda_{j-1}+\lambda_j,\]
    then we can write $\EE= F_1\boxtimes \dots \boxtimes F_n$. By Künneth formula we have
    $$ \Hom(\EE, g^\ast \EE) \cong \bigotimes_{i=1}^n \Hom(F_i, F_{g(i)}). $$ Since $[g] \in \sym_\lambda \backslash \sym_n$ and $[g] \neq [\id]$, we can find $i$'s such that $F_i \not\cong F_{g(i)}$, among which there must be some $i$ satisfying $\mu_H(F_i) \geqslant \mu_H(F_{g(i)})$. It follows from stability that $\Hom(F_i, F_{g(i)}) = 0$, hence $\Hom(\EE, g^\ast\EE)=0$; in particular $\EE \not\cong g^\ast\EE$, as desired.
\end{proof}

\begin{remark}\label{rem:counterexamplestab}
The assumption that $E_j$'s are pairwise non-isomorphic is indeed necessary, otherwise, there is always some equivariant subsheaf of $G_\lambda^\WW(E_1, \dots, E_k)$ with the same slope. For example, for $\lambda_1 + \lambda_2 = n$, we have that $$G_{(n)}^{\mathbf 1}(E)=E^{\boxtimes n}$$ is an equivariant subsheaf of 
\[G_{(\lambda_1,\lambda_2)}^{\mathbf 1}(E,E)\cong \bigoplus_{[g]\in (\sym_{\lambda_1}\times\sym_{\lambda_2})\backslash \sym_{n}}g^*\bigl(E^{\boxtimes \lambda_1}\boxtimes E^{\boxtimes \lambda_2}\bigr) \cong \bigl( E^{\boxtimes n}\bigr)^{\oplus \frac{n!}{\lambda_1!\lambda_2!}} \]
via the diagonal embedding, and both sheaves have the same slope.     
\end{remark}

\subsection{Stability of generalised tautological bundles}

We produce slope stable bundles on $S^{[n]}$ in two steps. First of all, we prove slope stability of generalised tautological bundles with respect to the semi-ample line bundle $H_{S^{[n]}}$. Then we deform it to a nearby ample line bundle.

Slope stability with respect to $H_{S^{[n]}}$ is closely related to $\mathfrak{S}_n$-equivariant slope stability with respect to $H_{S^n}$ by the following result

\begin{lemma}\label{gamma_stable}
	Assume $E$ is a torsion free sheaf on $S^{[n]}$ such that $(E)_{S^n}$ is $\mathfrak{S}_n$-equivariantly slope stable with respect to $H_{S^n}$ on $S^n$, then $E$ is slope stable with respect to $H_{S^{[n]}}$.
\end{lemma}

\begin{proof}
	Let $F\hookrightarrow E$ be a saturated subsheaf. Using the Stapleton functor \eqref{eqn:stapleton}, this induces an $\mathfrak{S}_n$-invariant subsheaf 
	\begin{equation*}
		(F)_{S^n} \hookrightarrow (E)_{S^n}.
	\end{equation*}
	But $(E)_{S^n}$ is $\mathfrak{S}_n$-equivariantly slope stable, hence
	\begin{equation*}
		\mu_{H_{S^n}}((F)_{S^n})<\mu_{H_{S^n}}((E)_{S^n}),
	\end{equation*}
    which is equivalent to the condition $\mu_{H_{S^{[n]}}}(F)<\mu_{H_{S^{[n]}}}(E)$ by \cite[Lemma 1.2]{stapleton_taut_2016}. It follows that $E$ is $H_{S^{[n]}}$-stable.
\end{proof}

\begin{proposition}\label{thm:Fstable}
	Assume that the locally free sheaves $E_1,\dots,E_k$ on $S$ are slope stable with respect to the ample line bundle $H$ and pairwise non-isomorphic, then the locally free sheaf $F_\lambda^{\mathbb{W}}(E_1,\ldots,E_k)$ on $S^{[n]}$ is slope stable with respect to $H_{S^{[n]}}$.
\end{proposition}

\begin{proof}
	 By \autoref{compare_stap} there is an isomorphism
	 \begin{equation*}
	 	\left(F_\lambda^{\mathbb{W}}(E_1,\ldots,E_k) \right) _{S^n}\cong G_\lambda^{\mathbb{W}}(E_1,\ldots,E_k).
	 \end{equation*}
	  Thus $F_\lambda^{\mathbb{W}}(E_1,\ldots,E_k)$ is slope stable with respect to $H_{S^{[n]}}$ by \autoref{gamma_stable}.
\end{proof}

We will deform $H_{S^{[n]}}$ to a nearby ample line bundle. For the ample line bundle to be independent of choices of $E_1, \dots, E_k$, we allow them to vary in families, and assume that their moduli spaces $M_1, \dots, M_k$ satisfy the following conditions

\begin{condition}
	\label{cond:moduli_dist}
	Given an ample line bundle $H$ on $S$,
	\begin{itemize}
		\item $M_j$ parametrizes \emph{locally free} sheaves on $S$ which are \emph{slope stable} with respect to $H$ for all $1 \leqslant j  \leqslant k$;
		\item $E_i \not\cong E_j$ for all $1 \leqslant i \neq j \leqslant k$, $[E_i] \in M_i$ and $[E_j] \in M_j$.
	\end{itemize}
\end{condition}

\begin{remark}\label{rmk:another-counter}
	Due to \autoref{rem:counterexamplestab}, we see that the second item in \autoref{cond:moduli_dist} is indeed necessary for our purpose. However, it is easy to be satisfied without changing the underlying moduli spaces $M_i$ (in other words, we can even achieve it if $M_i=M_j$ for some $i\neq j$). Recall in general that if $\mathcal F$ is a family of sheaves on $S$ parametrized by $M$, after tensoring the each of them with a fixed line bundle, the new class of sheaves are still parametrized by $M$. Therefore, by tensoring the sheaves parametrized by $M_j$'s with appropriate line bundles (for example, powers of $H$), the second item in \autoref{cond:moduli_dist} can always be achieved.
\end{remark}

The following theorem finally establishes the stability of generalised tautological bundles with respect to ample classes.

\begin{theorem}\label{same_div}
	Let $M_1, \dots, M_k$ be moduli spaces of $H$-stable sheaves on $S$ satisfying \autoref{cond:moduli_dist}. Then there exists an ample class $\tilde{H} \in \NS(S^{[n]})$ such that $F_\lambda^{\mathbb{W}}(E_1,\ldots,E_k)$ is $\mu_{\tilde{H}}$-stable for all $([E_1],\dots,[E_k])\in M_1\times\dots\times M_k$ simultaneously.
\end{theorem}

\begin{proof}
	For simplicity we denote 
	$$\cM = M_1 \times \dots \times M_k \qquad \text{and} \qquad m = ([E_1],\dots,[E_k])\in \cM.$$
	We observe that \autoref{cond:moduli_dist} guarantees that \autoref{thm:Fstable} holds for all $m \in \cM$. It remains to deform the semi-ample line bundle $H_{S^{[n]}}$ to the interior of the ample cone in such a way that all $F_\lambda^\WW(E_1,\dots, E_k)$ remain slope stable. The proof is essentially the same as that of \cite[Theorem 2.8]{reede_zhang_22}, in which the language of stability with respect to curve classes is used. We refer readers to \cite[Section 2.2]{greb_movable_2016} for the precise definition of this notion.

	More precisely, in the proof of \cite[Theorem 2.8]{reede_zhang_22}, we just need to replace the sheaf $E_x$ by $F_\lambda^{\mathbb{W}}(E_1,\ldots,E_k)$ and the surface $X$ by the product of moduli space $\cM$. We claim that there exists some auxiliary locally free sheaf $V \in \Coh(S^{[n]})$, such that there exists a surjective map
	\begin{equation}\label{eqn:same-V}
		V \twoheadrightarrow F_\lambda^{\mathbb{W}}(E_1,\ldots,E_k)^\vee
	\end{equation}
	for each $m \in \cM$.
	
	To prove this claim, we first assume that $M_j$ is a fine moduli space, or equivalently, a universal family $\cE_j \in \Coh(S \times M_j)$ exists for each $j=1, \dots, k$. The construction \eqref{eqn:family-F} yields a locally free sheaf $\cF_\lambda^\WW(\cE_1, \dots, \cE_k) \in \Coh(S^{[n]} \times \cM)$. And by \eqref{eqn:fiber-of-F} we have
	$$ \cF_\lambda^\WW(\cE_1, \dots, \cE_k)_m \cong F_\lambda^\WW(E_1, \dots, E_k). $$
	By taking its dual we obtain
	$$ \left( \cF_\lambda^\WW(\cE_1, \dots, \cE_k)^\vee \right)_m \cong \left( \cF_\lambda^\WW(\cE_1, \dots, \cE_k)_m \right)^\vee \cong F_\lambda^\WW(E_1, \dots, E_k)^\vee. $$
	Therefore the collection
	$$ \left\{ F_\lambda^{\mathbb{W}}(E_1,\ldots,E_k)^\vee \mid ([E_1],\dots,[E_k])\in M_1\times\dots\times M_k \right\} $$ 
	is bounded, and the locally free sheaf $V$ satisfying \eqref{eqn:same-V} exists; see e.g.\ \cite[Lemma 1.7.6]{huybrechts_moduli}.
	
	In general, the moduli spaces $M_1, \dots, M_k$ are not necessarily fine. For $j=1, \dots, k$, we replace the universal family $\cE_j$ by a quasi-universal family $\tilde{\cE_j}$, which is still a locally free sheaf such that $$\left( \tilde{\cE_j} \right)_{[E_j]} = E_j^{\oplus n_j}$$ for some positive integer $n_j$; see \cite[Section 4.6]{huybrechts_moduli}. Note that we have
	$$ \left( \cF_\lambda^\WW(\tilde{\cE}_1, \dots, \tilde{\cE}_k)^\vee \right)_m \cong F_\lambda^\WW(E_1^{\oplus n_1}, \dots, E_k^{\oplus n_k})^\vee \cong \left( F_\lambda^\WW(E_1, \dots, E_k)^\vee \right)^{\oplus n_1 \cdots n_k}, $$
	and the above reasoning also implies the existence of a locally free sheaf $V \in \Coh(S^{[n]})$ which admits surjective maps to all members of the collection
	$$ \left\{ \left( \cF_\lambda^\WW(\tilde{\cE}_1, \dots, \tilde{\cE}_k)^\vee \right)_m \ \middle|\ ([E_1],\dots,[E_k])\in M_1\times\dots\times M_k \right\}. $$ 
	It follows that there are also surjective maps from $V$ to the collection of direct summands
	$$ \left\{ F_\lambda^{\mathbb{W}}(E_1,\ldots,E_k)^\vee \mid ([E_1],\dots,[E_k])\in M_1\times\dots\times M_k \right\}, $$ 
	which concludes the claim.
	
	By taking the dual of the surjective maps in the claim, we obtain injective maps
	$$F_\lambda^{\mathbb{W}}(E_1,\ldots,E_k) \hookrightarrow V^\vee$$ 
	for all $m=([E_1],\dots,[E_k])\in \cM$. Therefore every subsheaf of the locally free sheaf $F_\lambda^{\mathbb{W}}(E_1,\ldots,E_k)$ for some $m=([E_1],\dots,[E_k])\in \cM$ is a subsheaf of $V^\vee$. 
	
	Let $\beta$ be a big movable curve class; see \cite[Definition 2.3]{greb_movable_2016}. The set $$ \{ c_1(G) \mid G \subseteq V^\vee \text{ such that } \mu_{\beta}(G) \geqslant c \} $$ is finite by \cite[Theorem 2.29]{greb_movable_2016}, and so is the subset $$ \{ c_1(G) \mid G \subseteq F_\lambda^{\mathbb{W}}(E_1,\ldots,E_k) \text{ for some } (E_1,\dots,E_k) \text{ such that } \mu_\beta(G) \geqslant c \}. $$
	
	The rest of the proof of \cite[Theorem 2.8]{reede_zhang_22} works unaltered.
\end{proof}

\section{Moduli spaces of generalised tautological bundles}\label{sect:moduli}

In this section, we will construct moduli spaces of generalised tautological bundles from moduli spaces of stable sheaves on $S$ under favorable conditions, and show that they give smooth connected components of moduli spaces of sheaves on $S^{[n]}$. We start with some numerical preparations.

\subsection{Self-extensions of box products}

We compute the self-extensions of box products, which will be used to compare dimensions of moduli spaces on $S$ and $S^{[n]}$.
For a coherent sheaf $\cF$, we use the shorthand $\End^i(\cF):=\Ext^i(\cF,\cF)$. Similarly, for an $\sym_n$-equivariant coherent sheaf $\cE$ on $S^n$, we write
\[
\End^i_{\sym_n}(\cE):=\Ext^i_{\sym_n}(\cE,\cE)\cong \Ext^i(\cE,\cE)^{\sym_n}
\]
for the self-extension groups in the equivariant category.  
We start with the following special case with a single sheaf $E$ as the input.

\begin{lemma}\label{lem:purebox}
Let $E\in \Coh(S)$ be a simple sheaf on $S$, and $W$ an irreducible $\sym_n$-representation. Then
\begin{itemize}
	\item $\End^0_{\sym_n}(E^{\boxtimes n}\otimes W)\cong \CC$ holds for arbitrary $W$;
	\item $\End^1_{\sym_n}\bigl(E^{\boxtimes n}\otimes W\bigr)\cong \End^1(E)$ holds if and only if the Young diagram of $W$ is rectangular.
\end{itemize}
\end{lemma}

\begin{proof}
For the first statement, by Künneth formula we have
\[
\End^0(E^{\boxtimes n})\cong \End^0(E)^{\otimes n}\cong \CC^{\otimes n} \cong \CC,
\]
where the middle isomorphism is due to the simplicity of $E$. It follows that
\[
\End^0_{\sym_n}(E^{\boxtimes n}\otimes W)\cong \Bigl(\End^0(E^{\boxtimes n})\otimes \End(W)  \Bigr)^{\sym_n}\cong \End_{\sym_n}(W)\cong \CC,
\]
where the last isomorphism is due to the irreducibility of $W$. 

For the second statement, we again use the Künneth formula to get 
\begin{align*}
\End^1(E^{\boxtimes n}) &\cong
\bigl(\End^1(E)\otimes \End^0(E)\otimes\dots\otimes \End^0(E)\bigr) \\
&\oplus \bigl(\End^0(E)\otimes \End^1(E)\otimes \dots\otimes \End^0(E)\bigr)\\ &\oplus \ \dots \\
&\oplus \bigl(\End^0(E)\otimes \End^0(E)\otimes\dots\otimes \End^1(E)\bigr)\,.
\end{align*}
The $\sym_n$-action on $\End^1(E^{\boxtimes n})$ induced by the canonical $\sym_n$-action on $E^{\boxtimes n}$ is given by permuting the above direct summands. Since $E$ is simple, we get an isomorphism of $\sym_n$-representations 
\[
\End^1(E^{\boxtimes n})\cong \End^1(E)\otimes V
\]
with $\sym_n$ acting trivially on $\End^1(E)$, and $V$ being the permutation representation of $\sym_n$. The latter splits as $V\cong \CC\oplus \rho$ where $\rho$ is the standard representation. Hence,
\begin{align*}
\End^1_{\sym_n}(E^{\boxtimes n}\otimes W)&\cong \Bigl(\End^1(E)\otimes V\otimes \End(W) \Bigr)^{\sym_n}\\&\cong \End^1(E)\otimes \bigl(V\otimes \End(W) \bigr)^{\sym_n},
\end{align*}
in which
\begin{align*}
	\bigl( V\otimes \End(W) \bigr)^{\sym_n} &\cong \bigl( \End(W) \bigr)^{\sym_n} \oplus \bigl( \rho\otimes \End(W) \bigr)^{\sym_n} \\
	&\cong \End_{\sym_n}(W) \oplus \End_{\sym_n}(W, \rho \otimes W).
\end{align*}
We observe again that $\End_{\sym_n}(W) \cong \CC$ since $W$ is irreducible. Moreover, an irreducible $\sym_n$-representation $W$ is a direct summand of $\rho\otimes W$ if and only if the Young diagram of $W$ is not rectangular; see \cite[top of p.\ 258]{Hamermesh}.
In other words, $\End_{\sym_n}(W, \rho \otimes W)=0$ if and only if $W$ is rectangular.
\end{proof}

We want extend the above result to the more general case with mixed inputs. For a collection $E_1, \dots, E_k \in \Coh(S)$, we make the following assumptions

\begin{condition}\label{cond:Homvanish}
There exists a set partition $\{1,\dots, k\}=I_1\uplus \dots \uplus I_\ell$ such that 
\begin{itemize}
\item $\Hom(E_j,E_j)=\mathbb C$ for all $1\leqslant j \leqslant k$;
\item $\Hom(E_i,E_j)=0$ for all $1 \leqslant \alpha \leqslant \ell$ and $i,j\in I_\alpha$ with $i \neq j$;
\item $\Hom(E_j,E_i)=\Ext^1(E_j,E_i)=0$ for all $1\leqslant \alpha<\beta \leqslant \ell$, $i\in I_\alpha$ and $j\in I_\beta$.
\end{itemize}
\end{condition}

We point out that \autoref{cond:Homvanish} in particular implies $E_1, \dots, E_k$ are pairwise non-isomorphic. The following result indicates that \autoref{cond:Homvanish} is not difficult to be satisfied. 

\begin{lemma}\label{lem:can-satisfy}
Let $M_1,\dots, M_\ell$ be moduli spaces of stable sheaves on $S$ and let $\{1,\dots, k\}=I_1\uplus \dots \uplus I_\ell$. Then by tensoring with suitable line bundles on $S$ if necessary, we can always achieve that every collection of pairwise non-isomorphic sheaves $E_1,\dots,E_k$ with $[E_i] \in M_\alpha$ if $i \in I_\alpha$ satisfies \autoref{cond:Homvanish}.    
\end{lemma}

\begin{proof}
The first item of \autoref{cond:Homvanish} holds because every stable sheaf is simple. 

For the second item, note that, for $i,j \in I_\alpha$ with $i \neq j$, the bundles $E_i$ and $E_j$ are non-isomorphic but of the same slope. Hence, there is no non-zero homomorphism between them.

For the third item, we can use boundedness of the classes of bundles parametrised by $M_\alpha$ and $M_\beta$. Indeed, twisting all sheaves parametrised by $M_\beta$ by the same sufficiently ample line bundle, we can assume that $\Ext^\ast(E_i,E_j\otimes \omega_S)$ has no higher cohomology for any $i\in I_\alpha$ and $j\in I_\beta$. By Serre duality, this gives the desired vanishing $\Hom(E_j,E_i)=\Ext^1(E_j,E_i)=0$.
\end{proof}

Under the above condition, we compute the self-extensions with mixed inputs.

\begin{proposition}\label{prop:end1}
Let $E_1,\dots, E_k$ be a collection of sheaves on $S$ satisfying \autoref{cond:Homvanish}. If the Young diagrams associated to the representations $W_1,\dots, W_k$ are all rectangular, then 
\[
\End^1_{\sym_n}\bigl(G^\WW_\lambda(E_1,\dots,E_k) \bigr)\cong \End^1(E_1)\oplus \dots \oplus \End^1(E_k).
\]
\end{proposition}

\begin{proof}
By the adjunction $\Ind \dashv \Res$, we have
\begin{equation}\label{eq:gradedEnd}
 \End^*_{\sym_n} \bigl( G^\WW_\lambda(E_1,\dots,E_k) \bigr) \cong \bigoplus_{[g]\in \sym_\lambda \backslash \sym_n}\Ext_{\sym_\lambda}^*\bigl(\EE\otimes \WW, g^*(\EE\otimes \WW)\bigr). 
\end{equation}
We show that $\Ext^1(\EE \otimes \WW, g^\ast (\EE \otimes \WW)) = 0$ for all $[g] \neq [\id]$. Forgetting the equivariant structure, $\EE\otimes \WW$ is just a direct sum of copies of $\EE$. Hence, it suffices to prove $\Ext^1(\EE, g^*\EE)=0$ for $g\notin \sym_\lambda$. We define $F_1,\dots, F_n$ as before by 
    \[F_i \coloneqq E_j\quad \text{for}\quad \lambda_1+\dots +\lambda_{j-1}< i \leqslant \lambda_1+\dots +\lambda_{j-1}+\lambda_j,\]
    then $\EE= F_1\boxtimes \dots \boxtimes F_n$. By Künneth formula we have
\begin{equation}\label{eq:KuennethF}\Ext^*\bigl(\EE, g^*\EE\bigr)\cong \bigotimes_{i=1}^n\Ext^*(F_i,F_{g(i)}).\end{equation}
The condition $g\notin \sym_\lambda$ means that there is some $i\in \{1,\dots,n\}$ with $F_i$ and $F_{g(i)}$ standing for different $E_j$'s. We distinguish two cases.

\textsc{Case 1.} Assume that all $(F_i,F_{g(i)})$ appearing on the right side of \eqref{eq:KuennethF} are of the form $(E_j,E_{j'})$ with $j,j'\in I_\alpha$ for some $1\leqslant \alpha \leqslant \ell$. Since $g \notin \sym_\lambda$, there must be at least two different values of $i$ with the corresponding $j \neq j'$, for which $\Ext^0(F_i, F_{g(i)}) = 0$ by the second item of \autoref{cond:Homvanish}. This means that the potentially non-zero part of the tensor product in \eqref{eq:KuennethF} starts in degree 2. In particular, $\Ext^1(\EE, g^*\EE)=0$.

\textsc{Case 2.} Otherwise, there exists some $i$ such that $(F_i, F_{g(i)})$ are of the form $(E_j, E_{j'})$, with $j \in I_\alpha$ and $j' \in I_{\alpha'}$ for $\alpha \neq \alpha'$. Among such $i$'s there must be one with the corresponding $\alpha > \alpha'$, which implies that $\Ext^\ast(F_i, F_{g(i)})$ vanishes in degree $0$ and $1$ by the third item of \autoref{cond:Homvanish}. Hence, \eqref{eq:KuennethF} again gives $\Ext^1(\EE, g^\ast\EE)=0$.

In summary, the only summand on the right side of \eqref{eq:gradedEnd} contributing a non-vanishing component of degree $1$ is given by $[g]=[\id]$. Hence, we have
\begin{align*}
	&\End^1_{\sym_n}(G^\WW_\lambda(E_1,\dots,E_k)) \\
	\cong\ &\End^1_{\sym_\lambda}\bigl((E_1^{\boxtimes \lambda_1}\otimes W_1) \boxtimes \dots \boxtimes (E_k^{\boxtimes \lambda_k}\otimes W_k)\bigr) \\
	\cong\ &\bigl(\End^1_{\sym_{\lambda_1}}(E_1^{\boxtimes \lambda_1}\otimes W_1)\otimes \End^0_{\sym_{\lambda_2}}(E_2^{\boxtimes \lambda_2}\otimes W_2)\otimes\dots\otimes \End^0_{\sym_{\lambda_k}}(E_k^{\boxtimes \lambda_k}\otimes W_k)\bigr)\\
	\oplus\ &\bigl(\End^0_{\sym_{\lambda_1}}(E_1^{\boxtimes \lambda_1}\otimes W_1)\otimes \End^1_{\sym_{\lambda_2}}(E_2^{\boxtimes \lambda_2}\otimes W_2)\otimes\dots\otimes \End^0_{\sym_{\lambda_k}}(E_k^{\boxtimes \lambda_k}\otimes W_k)\bigr)\\
	\oplus\ &\dots \\
	\oplus\ &\bigl(\End^0_{\sym_{\lambda_1}}(E_1^{\boxtimes \lambda_1}\otimes W_1)\otimes \End^0_{\sym_{\lambda_2}}(E_2^{\boxtimes \lambda_2}\otimes W_2)\otimes\dots\otimes \End^1_{\sym_{\lambda_k}}(E_k^{\boxtimes \lambda_k}\otimes W_k)\bigr),
\end{align*}
where the second isomorphism above is due to the Künneth formula. By the first assumption of \autoref{cond:Homvanish}, we can apply \autoref{lem:purebox} to further simplify the above and obtain
\begin{align*}
&\phantom{\cong} \ \End^1_{\sym_n}(G^\WW_\lambda(E_1,\dots,E_k)) \\
&\cong \End^1_{\sym_{\lambda_1}}(E_1^{\boxtimes \lambda_1}\otimes W_1)\oplus \End^1_{\sym_{\lambda_2}}(E_2^{\boxtimes \lambda_2}\otimes W_2)\oplus\dots\oplus \End^1_{\sym_{\lambda_k}}(E_k^{\boxtimes \lambda_k}\otimes W_k)\\
&\cong \End^1(E_1)\oplus \End^1(E_2)\oplus\dots\oplus \End^1(E_k). \qedhere
\end{align*}
\end{proof}

\subsection{An isomorphism between moduli spaces}

We resume to study the relation between moduli spaces of stable sheaves on $S$ and $S^{[n]}$. We first establish a morphism between them.

\begin{proposition}\label{prop:morphism-exists}
	Assume that $M_1, \dots, M_k$ are moduli spaces of stable sheaves on $S$ satisfying \autoref{cond:moduli_dist}. Then for some ample class $\tilde{H}$ on $S^{[n]}$, there exists a morphism
	\begin{equation}\label{eqn:mor_between_moduli}
		\varphi_\lambda^\WW \colon M_1\times\dots\times M_k \longrightarrow M_{S^{[n]}}
	\end{equation}
	where $M_{S^{[n]}}$ is the moduli space of $\tilde{H}$-stable sheaves on $S^{[n]}$, such that $$\varphi_\lambda^\WW ([E_1], \dots, [E_k]) = F_\lambda^\WW(E_1, \dots, E_k)$$ for each point $([E_1], \dots, [E_k]) \in M_1 \times \dots \times M_k$.
\end{proposition}

\begin{proof}
	As in the proof of \autoref{same_div}, we still denote $\cM \coloneqq M_1 \times \dots \times M_k$. If $M_j$ is a fine moduli space for each $j=1, \dots, k$, we assume $\cE_j \in \Coh(S \times M_j)$ is a corresponding universal family. Under the assumptions in \autoref{cond:moduli_dist}, the construction \eqref{eqn:family-F} gives a locally free sheaf $\cF_\lambda^\WW(\cE_1, \dots, \cE_k)$, which can be viewed as a family of slope stable locally free sheaves on $S^{[n]}$ with respect to the same ample class $\tilde{H}$ by \autoref{same_div}. Therefore, we obtain a classifying morphism
\begin{align*}
	\varphi_\lambda^\WW \colon \ \mathcal{M} = M_1\times\dots\times M_k\ &\longrightarrow\ M_{S^{[n]}}, \\
	m=([E_1], \dots, [E_k])\ &\longmapsto\ F_\lambda^{\mathbb{W}}(E_1,\ldots,E_k).
\end{align*}

More generally, if the moduli spaces $M_1, \dots, M_k$ are not fine, we can construct the morphism $\varphi_\lambda^\WW$ \'{e}tale locally, where universal families exist. Then we can glue the local morphisms to get a global morphism by using \'{e}tale descent; see also \cite[Proposition 2.1]{MR4581141} for some more details in a similar construction.
\end{proof}

Our next goal is to show that $\varphi_\lambda^\WW$ is an isomorphism to a connected component of $M_{S^{[n]}}$ under mild assumptions. The following criterion will be helpful in showing that $\varphi_\lambda^\WW$ is injective on closed points.

\begin{proposition}\label{prop:inj_points}
	Let $M_1, \dots, M_k$ be moduli spaces satisfying \autoref{cond:moduli_dist}. Then for a pair of points $([E_1], \dots, [E_k]), ([E'_1], \dots, [E'_k]) \in M_1 \times \dots \times M_k$,
	\begin{equation}\label{eqn:inj-isom}
		G_\lambda^\WW(E_1, \dots, E_k) \cong G_\lambda^\WW(E'_1, \dots, E'_k)
	\end{equation}
	as $\sym_n$-equivariant sheaves if and only if $([E_1], \dots, [E_k])=([E'_1], \dots, [E'_k])$.
\end{proposition}

\begin{proof}
	We only need to prove the only if part, for which we do computation similar to those in the proof of \autoref{prop:end1}. By the adjunction $\Ind\dashv \Res$, we have
	\[
        \Hom_{\sym_n}(G_\lambda^\WW(E_1, \dots, E_k), G_\lambda^\WW(E'_1, \dots, E'_k)) 
        = \Hom_{\sym_\lambda}(\EE \otimes \WW, \bigoplus_{[g] \in \sym_\lambda\backslash\sym_n} g^\ast(\EE' \otimes \WW)).
    \]
    For each $[g] \in \sym_\lambda\backslash\sym_n$, the K\"unneth formula implies that 
    $$\Hom(\EE, g^\ast\EE') \cong \bigotimes_{i=1}^n \Hom(F_i, F'_{g(i)}).$$
    If $[g] \neq [\id]$, then there exists some $i$, such that $F_i = E_j$ and $F'_{g(i)} = E'_{j'}$ with $j \neq j'$. By the second item of \autoref{cond:moduli_dist} we get $\Hom(F_i, F'_{g(i)})=0$. It follows that $\Hom(\EE, g^\ast\EE')=0$ and moreover $\Hom(\EE \otimes \WW, g^\ast(\EE' \otimes \WW))=0$. Therefore
    \[
        \Hom_{\sym_n}(G_\lambda^\WW(E_1, \dots, E_k), G_\lambda^\WW(E'_1, \dots, E'_k)) 
        = \Hom_{\sym_\lambda}(\EE \otimes \WW, \EE' \otimes \WW),
    \]
    which has to be non-zero if \eqref{eqn:inj-isom} holds.
    By Künneth formula again we have
    $$ \Hom(\EE, \EE') = \bigotimes_{i=1}^n \Hom(F_i, F'_i) = \bigotimes_{j=1}^k \Hom(E_j, E'_j)^{\otimes \lambda_j}. $$
    Therefore \eqref{eqn:inj-isom} implies that $E_j \cong E'_j$ for all $j = 1, \dots, k$.
\end{proof}

Finally, we are ready to state the main result of this section.

\begin{theorem}\label{thm:moduli-main-thm}
	Let $M_1, \dots, M_k$ be connected components of moduli spaces of stable sheaves on $S$. Assume they are smooth and projective, and satisfy \autoref{cond:moduli_dist}, such that \autoref{cond:Homvanish} holds with $\ell=k$ for every choice of $[E_1] \in M_1, \dots, [E_k] \in M_k$. Also assume that the Young diagrams associated to the representations $W_1, \dots, W_k$ are all rectangular. Then $\varphi_\lambda^\WW$ is an isomorphism from $\cM = M_1 \times \dots \times M_k$ to a connected component of $M_{S^{[n]}}$.
\end{theorem}

\begin{proof}
	Under \autoref{cond:moduli_dist}, the morphism $\varphi_\lambda^\WW \colon \cM \to M_{S^{[n]}}$ was constructed in \autoref{prop:morphism-exists}, and \autoref{prop:inj_points} applies. Together with the fact that the functor $\Psi$ in \eqref{eqn:Fdef} is an equivalence, we conclude that $F_\lambda^\WW(E_1, \dots, E_k) \cong F_\lambda^\WW(E'_1, \dots, E'_k)$ if and only if $E_j \cong E'_j$ for $j=1, \dots, k$. Therefore $\varphi_\lambda^\WW$ is injective on closed points.
	
	Moreover, \autoref{prop:end1} holds under the assumption of \autoref{cond:Homvanish}. Together with the fact that the functor $\Psi$ in \eqref{eqn:Fdef} is an equivalence, we get
	\begin{align*}
		&\dim \End^1(F_\lambda^\WW(E_1, \dots, E_k)) = \dim \End^1_{\sym_n} (G_\lambda^\WW(E_1, \dots, E_k)) \\
		=\ &\dim \End^1(E_1) + \dots + \dim \End^1(E_k),
	\end{align*}
	which means indeed that
	$$ \dim T_{\varphi_\lambda^\WW ([E_1], \dots, [E_k])} M_{S^{[n]}} = \dim T_{([E_1], \dots, [E_k])}\cM. $$

	We conclude by \cite[Lemma 1.6]{reede_examples_2021} that $\varphi_\lambda^\WW$ is an isomorphism from $\cM$ to a connected component of $M_{S^{[n]}}$, as desired.
\end{proof}

This theorem gives an explicit description of many smooth projective connected components of $M_{S^{[n]}}$.

\begin{remark}
	Let us quickly explain the need for the assumptions made in \autoref{thm:moduli-main-thm}.
    
	If the second item in \autoref{cond:moduli_dist} is not met, then stability of $F_\lambda^{\WW}(E_1,\dots,E_k)$ can fail which leads to $\varphi_\lambda^\WW$ not being well-defined everywhere; see \autoref{rem:counterexamplestab}. Even if $\varphi_\lambda^\WW$ exists as a rational map, it might not be generically injective. To see this phenomenon, assume that $M_1=M_2\eqqcolon M$ (not just as an isomorphism of varieties, but really as moduli spaces parametrizing the same class of stable bundles), $\lambda=(\lambda_1,\lambda_2)$ with $\lambda_1=\lambda_2$, and $W_1\cong W_2$. Then, for all $[E_1],[E_2]\in M$, we have $F_\lambda^{\WW}(E_1,E_2)\cong F_\lambda^{\WW}(E_2,E_1)$. This means that the morphism \[\varphi_{\lambda}^{\WW}\colon (M\times M)\setminus \Delta \to M_{S^{[n]}}\] factors through the $\sym_2$-quotient in this case.
    
	If the Young diagrams of $W_j$'s are not rectangular, the proofs of \autoref{lem:purebox} and \autoref{prop:end1} show that the image of $\varphi_\lambda^\WW$ is of much smaller dimension than the component of $M_{S^{[n]}}$ that contains it (or $M_{S^{[n]}}$ is non-reduced). In other words, generalised tautological bundles only occupy a proper subvariety of such a component.

We are not sure whether the first assumption of \autoref{cond:moduli_dist}, that the moduli spaces only parametrize locally free sheaves, is really neccessary. We use this assumption in some steps of our arguments. However, we do not rule out that with some more technical work, one can extend \autoref{thm:moduli-main-thm} from locally free to torsion free sheaves; see \cite[Section 2]{ogrady_many} for partial results in this direction in the case $\lambda=(1,\dots, 1)$.

\end{remark}

\section{First Chern classes}\label{sect:Chern}

In this section, we compute an explicit formula for the first Chern class of generalised tautological bundles $c_1\bigl(F_\lambda^\WW(E_1,\dots, E_k)\bigr)$. Since $F=F_\lambda^\WW(E_1,\dots, E_k)$ is locally free, to compute its first Chern class, it is enough to understand its restriction to an open subset whose complement is of codimension at least $2$. 

\subsection{Strategy}
We consider the open subscheme
\begin{equation*}
	S^{(n)}_\ast \coloneqq \left\lbrace \sum x_i\in S^{(n)} \ \middle|\ \lvert \{x_1,\dots,x_n\} \rvert \geqslant n-1 \right\rbrace \subset S^{(n)}
\end{equation*}
which is the locus of unordered tuples in which at most two points coincide. 
Taking the inverse image of $S^{(n)}_*\subset S^{(n)}$ under all morphisms in \eqref{eqn:maindiagram}, we get the following commutative diagram
\begin{equation}\label{eq:McKaydiagopen}
	\begin{tikzcd}
		I^nS_\ast \arrow[r,"p"]\arrow[d,swap,"q"] & S^n_\ast\arrow[d,"\pi"]\\
		S^{[n]}_\ast\arrow[r,"\mu"']& S^{(n)}_\ast.
	\end{tikzcd}
\end{equation}
We denote the restriction of the maps of \eqref{eqn:maindiagram} to the $(\_)_*$-decorated open subsets still by the same symbols as before and hope that this does not cause confusion (decorating also the morphisms by $(\_)_*$ would surely cause confusion with push-forwards). However, if we restrict a sheaf, say $G$, on any of the four schemes of \eqref{eqn:maindiagram} to the open subscheme, we make this explicit by writing the restriction as $G_*$. Note that the codimension of the complement of $S^{[n]}_*$ in $S^{[n]}$ is 2, while the complements of the other three open subschemes are of codimension 4. 

The morphism $p\colon I^nS\to S^n$ is the blow-up in the big diagonal; see \cite[Prop.\ 3.4.2]{Hai}. The advantage of the restriction $p\colon I^nS_*\to S^n_*$ is that the big diagonal in $S^n_*$ becomes a disjoint union of smooth subvarieties $\Delta=\coprod\limits_{1\leqslant i < j \leqslant n} \Delta_{ij*}$, where $\Delta_{ij*}$ denotes the restriction of the pairwise diagonal
\begin{equation*}
	\Delta_{ij}=\left\lbrace (x_1,\ldots,x_n)\in S^n\,\,|\,\, x_i=x_j\, \right\rbrace 
\end{equation*}
to $S^n_*$. We denote the exceptional divisor by $E_*\subset I^nS_{*}$ and note that
\begin{equation*}
	E_*=\coprod\limits_{1\leqslant i < j \leqslant n} E_{ij*}
\end{equation*}
where $E_{ij*}=p^{-1}(\Delta_{ij*})\rightarrow \Delta_{ij*}$ is a $\mathbb{P}^1$-bundle. 
Note that $E_{ij*}$ is the fixed point locus of the action of $\mathfrak{S}_{ij}=\left\langle (ij) \right\rangle \subset \mathfrak{S}_n$, the subgroup generated by the transposition $(ij)$, on $I^nS_*$.

Note that $S^{[n]}_{*}$ is still the $\mathfrak{S}_n$-quotient of $I^nS_*$. However, in contrast to its action on $I^nS$, the alternating normal subgroup $\mathfrak{A}_n \lhd \mathfrak{S}_n$ acts freely on $I^nS_*$ as the only elements with fixed points in $I^nS_*$ are the transpositions, which are not in $\mathfrak{A}_n$. Thus, the morphism $q: I^nS_{*}\rightarrow S^{[n]}_{*}$ factorises as 
\begin{equation*}
	\begin{tikzcd}
		I^nS_{*} \arrow[r,"q_1"] & T=I^nS_{*}/\mathfrak{A}_n\arrow[r,"q_2"]& S^{[n]}_{*}	
	\end{tikzcd}
\end{equation*}
where $q_1: I^nS_* \rightarrow T$ is an \'{e}tale morphism of degree $\frac{n!}{2}$, and $q_2 : T \rightarrow S^{[n]}_{*}$ is a double cover, with covering group $H=\mathfrak{S}_n/\mathfrak{A}_n$, ramified over the boundary divisor 
\begin{equation*}
	D_*=D\cap S^{[n]}_*\quad,\quad D=\left\lbrace [Z]\in S^{[n]}\,\,|\,\, |\mathrm{supp}(Z)|\leqslant n-1\, \right\rbrace. 
\end{equation*}
We set  $E'=q_2^{-1}(D_*)=q_1(E_*)$, and write $\fa_H$ for the alternating representation of the order 2 group $H$.

For the convenience of readers, we summarise all notations in the diagram
\begin{equation*}
	\begin{tikzcd}[row sep=large]
		& & & E_{ij\ast} \ar[rr,"\PP^1"]\ar[d,"j"']\ar[dll, "l"']\ar[ddlll, bend right=50, "r"'] & & \Delta_{ij\ast} \ar[d] \\
		& E_\ast \ar[rr,hook,"k"] \ar[dl,"\pi_1"'] & & I^nS_\ast \ar[rr,"p"] \ar[dl,"(\_)/\mathfrak{A}_n=q_1"'] \ar[dd,"q=(\_)/\sym_n"] & & S^n_\ast \ar[dd,"\pi"] \\
		E'\ar[dr,"\pi_2"'] \ar[rr,hook,"i"] & & T \ar[dr,"(\_)/H=q_2"'] & & & \\
		& D_\ast \ar[rr,hook] & & S^{[n]}_\ast \ar[rr,"\mu"] & & S^{(n)}_\ast
	\end{tikzcd}
\end{equation*}

By \cite[Theorem 4.1]{kuzper} (see also \cite[Theorem 5.1]{colpol}) there is a semi-orthogonal decomposition
\begin{equation}\label{sod}
	\mathrm{D}^b_{H}(T)=\left\langle i_{*}\mathrm{D}^b(E')\otimes\fa_H,\,q_2^{*}\mathrm{D}^b(S^{[n]}_{*}) \right\rangle 
\end{equation}
that we will use in the following.

\begin{lemma}\label{lem:KPses}
	Let $\cG$ be an $\sym_n$-equivariant locally free sheaf on $I^nS_*$. Then there is an exact sequence
	\begin{equation}\label{eqn:push-pull}
		\begin{tikzcd}
			0\arrow[r]& q^{*}q_{*}^{\mathfrak{S}_n}\cG \arrow[r] & \cG \arrow[r]& \mathcal{C} \arrow[r] & 0	\quad, \quad \cC= \bigoplus\limits_{1\leqslant i < j \leqslant n} \left( \cG_{|E_{ij*}}\otimes\fa_{ij} \right)^{\mathfrak{S}_{ij}}
		\end{tikzcd}
	\end{equation}
in $\Coh(I^nS_*)$ where $\fa_{ij}$ denotes the alternating representation of $\mathfrak{S}_{ij}$.
\end{lemma}

\begin{proof}
Set $M:=(q_1)_*^{\mathfrak{A}_n}\cG$, then $M$ is locally free as $\cG$ is and $q_1$ is \'{e}tale. By \eqref{sod}, we get an exact sequence
	\begin{equation}\label{eq:Tses}
	\begin{tikzcd}
		0\arrow[r]& q_2^{*}(q_2)_{*}^{H}M \arrow[r] & M \arrow[r]& \mathcal{C}' \arrow[r] & 0	
	\end{tikzcd}
\end{equation}
in $\Coh_H(T)$ where $\mathcal{C}'=i_{*}\left( \left(i^{*}M\otimes \fa_H \right)^H\otimes \fa_H\right)$. Note that \eqref{sod} a priori gives an exact triangle in $\mathrm{D}^b_H(T)$ and the pull-back $i^*$ in $\cC'$ needs to be derived. But, as $M$ is a locally free sheaf, the pull-back $i^*M$ does not need to be derived and the triangle in the derived category reduces to an exact sequence of sheaves. Further note that, 
since we are only interested in the underlying non-equivariant sheaves, we can write simply $\mathcal{C}'=i_{*}\bigl((i^{*}M\otimes \fa_H)^H\bigr)$, which means that we can ignore the tensor product by the non-trivial character after taking the invariants, but not before taking the invariants.

We show that the pull-back of \eqref{eq:Tses} along the flat morphism $q_1$ gives the asserted short exact sequence in $\Coh_{\sym_n}(I^nS_*)$. Since $q_1$ is the quotient by the free $\mathfrak{A}_n$-action, 
\begin{equation}\label{eq:qM}
	q_1^{*}M\cong q_1^{*}(q_1)_{*}^{\mathfrak{A}_n}\cG\cong \cG\,. 
\end{equation}
Furthermore, by functoriality of push-forwards and pull-backs,
\[
	q_1^{*}q_2^{*}\left(q_2 \right)_*^H M \cong (q_2\circ q_1)^{*}\left(q_2 \right)_*^H\left(q_1 \right)_*^{\mathfrak{A}_n} \cG \cong (q_2\circ q_1)^{*}\left(q_2 \circ q_1 \right)_*^{\mathfrak{S}_n} G \cong q^{*}q_{*}^{\mathfrak{S}_n} \cG 
\]
To compute $\cC:=q_1^{*}\mathcal{C}'$, we first note that by flat base change 
along the cartesian diagram 
	\begin{equation}\label{eq:Ediag}
	\begin{tikzcd}
			E_* \arrow[r,"k"]\arrow[d,swap,"\pi_1"] & I^nS_*\arrow[d,"q_1"]\\
		 E'\arrow[r,"i"]& T
	\end{tikzcd}
\end{equation}
we get $\cC\cong k_*\pi_1^* \bigl((i^{*}M\otimes \fa_H)^H\bigr)$.
In particular, $\cC$ is scheme-theoretically supported on $E_*=\coprod E_{ij*}$ which gives 
\[
\cC\cong \bigoplus \cC_{\mid E_{ij*}}\,.
\]
Now, for some fixed $1\leqslant i<j\leqslant n$, we enlarge \eqref{eq:Ediag} to the diagram
\begin{equation*}
	\begin{tikzcd}
		E_{ij*}\arrow[rr, bend left, "j"]\arrow[r, "l"]\arrow[dr,swap,"r"]&	E_* \arrow[r,"k"]\arrow[d,swap,"\pi_1"] & I^nS_*\arrow[d,"q_1"]\\
		& E'\arrow[r,"i"]& T
	\end{tikzcd}
\end{equation*}
All morphisms of the outer part of this diagram, i.e.\ in the commutative diagram
\begin{equation*}
	\begin{tikzcd}
		E_{ij*} \arrow[r,"j"]\arrow[d,swap,"r"] & I^nS_*\arrow[d,"q_1"]\\
		 E'\arrow[r,"i"]& T\,,
	\end{tikzcd}
\end{equation*}
are $\sym_{12}$-equivariant when we consider $\sym_{12}$ acting on the top row as a subgroup of $\sym_n$ and on the bottom row via the natural isomorphism $\sym_{12}\hookrightarrow \sym_n\to \sym_n/\mathfrak A_n=H$. Hence, 
\begin{align*}
\cC_{\mid E_{ij*}}=j^*q_1^*\cC'\cong j^*q_1^* i_{*}\bigl((i^{*}M\otimes \fa_H)^H\bigr)&\cong r^*i^* i_{*}\bigl((i^{*}M\otimes \fa_H)^H\bigr)\\
&\cong r^*\bigl((i^{*}M\otimes \fa_H)^H\bigr)\\
&\cong (j^*q_1^{*}M\otimes \fa_{12})^{\sym_{12}}\\
&\cong \bigl(\cG_{\mid E_{ij}}\otimes \fa_{12})^{\sym_{12}}
\end{align*}
where the last isomorphism uses \eqref{eq:qM}.
\end{proof}

\begin{remark}
	With just a little bit more care than in the above proof, one can turn \eqref{eqn:push-pull} into a short exact sequence in 
$\Coh_{\sym_n}(I^nS_*)$, i.e.\ a sequence of equivariant sheaves, by writing 
\[
\cC=\Ind_{\sym_2\times \sym_{n-2}}^{\sym_n} \bigl( ( \cG_{|E_{12}}\otimes\fa_{12} )^{\mathfrak{S}_{12}}\otimes \fa_{12}\bigr)\,.
\]
However, the linearisations of the sheaves do not play a role for their Chern classes, so we opted to only prove the slightly weaker non-equivariant statement. 
\end{remark}

\begin{proposition}\label{prop:c1gen}
Let $G$ be a $\sym_n$-equivariant locally free sheaf on $S^n$ such that $c_1(G)=B_{S^n}$ for some $B\in \Pic S$. Then, setting $\delta=[D]/2$, we have 
\[
c_1(\Psi(G))= B_{S^{[n]}}- \rk((G_{\mid \Delta_{12}}\otimes \fa_{12})^{\sym_{12}})\cdot \delta 
\]
\end{proposition}

\begin{proof}
As alluded to at the start of this section, we can compute $c_1(\Psi(G))$ on $S^{[n]}_*$ instead of $S^{[n]}$. The reason is that the complement of $S^{[n]}_*$ of $S^{[n]}$ is of codimension 2. Hence, restriction gives an isomorphism $\Pic(S^{[n]})\cong \Pic(S^{[n]}_*)$, and $c_1(\Psi(G))$ and $c_1(\Psi(G)_*)=c_1(\Psi(G_*))$ are identified by this isomorphism. We will not distinguish between $\Delta_{ij*}$, $E_{ij*}$, $G_*$ and $\Delta_{ij}$, $E_{ij}$, $G$ in this proof to lighten the notation.

Applying \autoref{lem:KPses} with $\cG=p^*(G)$, we get the formula
\begin{equation}\label{eq:qChern}
q^*c_1(\Psi(G))= p^* c_1(G) - c_1(\cC)\,.
\end{equation}
By the commutativity of \eqref{eq:McKaydiagopen}, we get 
\[
p^* c_1(G)=p^*(B_{S^n})=q^*(B_{S^{[n]}})\,.
\]
As $\cC$ is supported on the disjoint divisors $E_{ij}$, we have 
\[
c_1(\cC)=\sum_{1\leqslant i<j\leqslant n} \rk(\cC_{\mid E_{ij}}) E_{ij}\,.
\]
By $\sym_n$-equivariance, all the $\rk(\cC_{\mid E_{ij}})$ are the same. Furthermore, $\displaystyle q^*\delta=\sum_{1\leqslant i<j\leqslant n} E_{ij}$. Hence,
\[
c_1(\cC)=\rk(\cC_{\mid E_{12}}) q^*\delta\,.
\]
We have
\[
\rk(\cC_{\mid E_{12}})= \rk\bigl(((p^*G)_{\mid E_{12}}\otimes \fa_{12})^{\sym_{12}} \bigr)=\rk((G_{\mid \Delta_{12}}\otimes \fa_{12})^{\sym_{12}})\,. 
\]
Combining all of the above, we can rewrite \eqref{eq:qChern} as
\[
q^*c_1(\Psi(G))= q^*(B_{S^{[n]}})  - \rk((G_{\mid \Delta_{12}}\otimes \fa_{12})^{\sym_{12}}) \cdot q^*\delta\,.
\]
The assertion follows from the injectivity of $q^*$.
\end{proof}

\begin{remark}
	By the $\sym_n$-equivariance of $G$, we have \[\rk((G_{\mid \Delta_{12}}\otimes \fa_{12})^{\sym_{12}})=\rk((G_{\mid \Delta_{ij}}\otimes \fa_{ij})^{\sym_{ij}})\] for any $1\leqslant i<j\leqslant n$. Hence, we could replace $(\_)_{12}$ by any other double index $(\_)_{ij}$ in the formula of \autoref{prop:c1gen}.
\end{remark}

\subsection{Computation}

We now consider the case that $G=G_{\lambda}^\WW(E_1,\dots, E_k)$ for some partition $\lambda=(\lambda_1,\dots,\lambda_k)$ of $n$ and some irreducible presentation $\WW=W_1\otimes \dots\otimes W_k$ of $\sym_{\lambda}= \sym_{\lambda_1}\times \dots\times \sym_{\lambda_k}$ with $W_i$ an irreducible representation of $\sym_{\lambda_i}$. 

Recall the notation $\EE:=E_1^{\boxtimes \lambda_1}\boxtimes \dots \boxtimes E_k^{\boxtimes \lambda_k}$, which gives
\begin{equation}\label{eq:Indsummands}
G=G_{\lambda}^{\WW}(E_1,\dots, E_k)=\Ind_{\sym_\lambda}^{\sym_n}(\EE\otimes \WW)\cong \bigoplus_{[g]\in \sym_\lambda\backslash \sym_n} g^*(\EE\otimes \WW)
\end{equation}
where the direct sum decomposition only holds for the underlying non-equivariant sheaf.
Note that the coset $\sym_\lambda\backslash \sym_n$ is in bijection to the set of set-valued partitions of $\{1,\dots,n\}$ with underlying numerical partition $\lambda$ via 
$g\mapsto \bigl(g^{-1}(I_1), \dots, g^{-1}(I_k)\bigr)$ where
\[
I_\ell=\bigl[1+\sum_{i=1}^{\ell-1}\lambda_i, \sum_{i=1}^\ell \lambda_i\bigr]=\bigl\{j \mid \text{at the $j$-th position of the box product $\EE$ there is }E_\ell\bigr\}
\]
We introduce some notation which enables us to state the formulas in a more compact way. We write 
\[
p_{\lambda}\coloneqq [\sym_n:\sym_{\lambda}]=\frac{|\sym_n|}{|\sym_{\lambda}|}=\frac{n!}{\lambda_1!\cdots \lambda_k!}\,.
\]
Furthermore, for $1\leqslant 1\leqslant k$, we consider the partition $\lambda(\bar i)\coloneqq(\lambda_1,\dots, \lambda_i-1, \dots, \lambda_k)$ of $n-1$, where the sequence has to be reordered if necessary to be non-decreasing. Similarly, for $1\leqslant i\leqslant j\leqslant k$ we consider the partition of $n-2$ given by 
\[
\lambda(\bar i, \bar j)\coloneqq \begin{cases}
(\lambda_1,\dots, \lambda_i-2, \dots, \lambda_k) \quad &\text{if $i=j$}\\    
(\lambda_1,\dots, \lambda_i-1, \dots,\lambda_j-1,\dots  \lambda_k) \quad &\text{if $i\neq j$}
\end{cases}\,.
\]
Again, the sequence has to be reordered if necessary to become non-decreasing. We also write
\[
p_{\lambda(\bar i)}\coloneqq[\sym_{n-1}:\sym_{\lambda(\bar i)}]\quad ,\quad p_{\lambda(\bar i, \bar j)}\coloneqq [\sym_{n-2}:\sym_{\lambda(\bar i, \bar j)}]\,.
\]
Note that 
\begin{equation}\label{eq:pnumber}
\begin{aligned}
p_{\lambda(\bar i)}&=|\bigl\{[g]\in \sym_\lambda\backslash \sym_n\mid g(1)\in I_i \bigr\}|\,,\\ p_{\lambda(\bar i, \bar j)}&=|\bigl\{[g]\in \sym_\lambda\backslash \sym_n\mid g(1)\in I_i,\, g(2)\in I_j \bigr\}|\,.
\end{aligned}
\end{equation}
Furthermore, we write
\[
w_i\coloneqq \dim W_i\quad,\quad w\coloneqq \dim \WW=\prod_{i=1}^k w_i \,.
\]
and
\[
r_i\coloneqq \rk E_i\quad, \quad s\coloneqq \rk\EE=\prod_{i=1}^k r_i^{\lambda_i}\,.
\]

\begin{lemma}\label{lem:c1G} We have
$c_1(G)=sw\Bigl(\sum_{i=1}^k\frac{p_{\lambda(\bar i)}}{r_i} c_1(E_i)\Bigr)_{S^n}$\,.    
\end{lemma}

\begin{proof}
The $\sym_n$-linearisation of $G$ does not play a role for its first Chern class. Hence, we can treat $\WW$ just as a $w$-dimensional vector space, and get $c_1(G)=w\cdot c_1\bigr(G_{\lambda}^{\mathbf 1}(E_1,\dots, E_k)\bigl)$.
Hence, it suffices to prove the formula
\begin{equation}\label{eq:ChernG1}
c_1\bigr(G_{\lambda}^{\mathbf 1}(E_1,\dots, E_k)\bigl)= s\Bigl(\sum_{i=1}^k\frac{p_{\lambda(\bar i)}}{r_i} c_1(E_i)\Bigr)_{S^n}\,.
\end{equation}
Recall that there is the non-equivariant direct sum decomposition
\[
G_{\lambda}^{\mathbf 1}(E_1,\dots, E_k)\cong \bigoplus_{[g]\in \sym_\lambda\backslash \sym_n} g^*\EE    \,.
\]
For $i=1,\dots, k$, there are $p_{\lambda(\bar i)}$ summands of the form $g^*\EE\cong E_i\boxtimes \cF$ where $\cF$ is some permutation of the box factors of $E_1^{\boxtimes \lambda_1}\boxtimes\dots\boxtimes E_i^{\boxtimes \lambda_i-1}\boxtimes\dots\boxtimes E_k^{\boxtimes \lambda_k}$; see \eqref{eq:pnumber}. In particular, $\rk \cF=\frac s{r_i}$, and the formula for the first Chern class of tensor products gives
\[
c_1(E_i\boxtimes \cF)=p_1^*\bigl( \frac s{r_i} c_1(E_i) \bigr)+ \sum_{j=2}^np_j^*(\text{something})\,. 
\]
Summing up all these summands for the various $i=1,\dots k$ gives
\[
c_1\bigr(G_{\lambda}^{\mathbf 1}(E_1,\dots, E_k)\bigl)= p_1^*\Bigl(s\bigl(\sum_{i=1}^k\frac{p_{\lambda(\bar i)}}{r_i} c_1(E_i)\bigr)\Bigr)+ \sum_{j=2}^np_j^*(\text{something})\,.
\]
By the $\sym_n$-equivariance of $G_{\lambda}^{\mathbf 1}(E_1,\dots, E_k)$, the value of \emph{something} in $p_j^*$, has to be the same for all $j$, including $j=1$. Hence, it is $s\bigl(\sum_{i=1}^k\frac{p_{\lambda(\bar i)}}{r_i} c_1(E_i)\bigr)$, which shows \eqref{eq:ChernG1}.
\end{proof}

Let $\lambda_i\geqslant 2$. We choose any transposition $\tau$ in $\sym_{\lambda_i}$ and identify $\sym_2$ with the subgroup $\langle \tau\rangle \subset \sym_{\lambda_i}$. We denote by 
\[
\alpha_i:=\dim\Hom_{\sym_2}(\mathbf 1, \Res_{\sym_{\lambda_i}}^{\sym_2} W_i)\quad, \quad \beta_i:=\dim\Hom_{\sym_2}(\fa, \Res_{\sym_{\lambda_i}}^{\sym_2} W_i)
\]
the multiplicities of the trivial and the sign representation of $\sym_2$ as direct summands of $\Res_{\sym_{\lambda_i}}^{\sym_2} W_i$. The numbers $\alpha_i$ and $\beta_i$ are independent of the chosen transposition $\tau$. 
With this notation, we have
\[
\Res_{\sym_{\lambda_i}}^{\sym_2} W_i\cong \mathbf 1^{\oplus \alpha_i}\oplus \fa^{\oplus \beta_i}\,.
\]

\begin{lemma}\label{lem:rkG} We have
\begin{align*}&\rk((G_{\mid \Delta_{12}}\otimes \fa_{12})^{\sym_{12}})\\
=&sw \left(\sum_{1\leqslant i<j\leqslant k} p_{\lambda}(\bar i, \bar j)  + \sum_{i \text{ with }\lambda_i\geqslant 2} \frac{p_\lambda(\bar i, \bar i)}{r_i^2 w_i}\Bigr(\alpha_i\binom{r_i}2 + \beta_i\binom{r_i+1}2   \Bigl) \right)\,.
\end{align*}    
\end{lemma}

\begin{proof}
Restricting \eqref{eq:Indsummands} to the $\sym_{12}$-invariant subvariety $\Delta_{12}\subset S^n$ and tensoring by the sign representation $\fa_{12}$, we get 
\begin{equation}\label{eq:IndDeltasummands}
G_{\mid \Delta_{12}}\otimes \fa_{12} \cong \bigoplus_{[g]\in \sym_\lambda\backslash \sym_n} \bigl(g^*(\EE\otimes\WW)\bigr)_{\mid \Delta_{12}}\otimes \fa_{12} \,.
\end{equation}
In the following, we will often just write $g^*(\EE\otimes\WW)_{\mid \Delta_{12}}$ instead of $\bigl(g^*(\EE\otimes\WW)\bigr)_{\mid \Delta_{12}}$. This means that we leave out one pair of brackets for better readability, tacitly understanding that the pull-back along $g$ has to be applied before restricting to $\Delta_{12}$ since $\Delta_{12}\subset S^n$ is not stable under the $G$-action.

Let $\tau=(1\,\, 2)$ denote the non-trivial element of $\sym_{12}$. 
If $g(1)\in I_i$ and $g(2)\in I_j$ with $i\neq j$, then $g\tau$ represents a different coset than $g$ in $\sym_{\lambda}\backslash \sym_n$ and the $\sym_{12}$-action on $G_{\mid \Delta_{12}}\otimes \fa_{12}$ permutes the two corresponding summands in  
\eqref{eq:IndDeltasummands}. This contributes to the $\sym_{12}$-invariant subbundle of $G_{\mid \Delta_{12}}\otimes \fa_{12}$ as 
\begin{equation} \label{eq:twosummandsinva}
\Bigl(\bigl(g^*(\EE\otimes \WW)_{\mid_{\Delta_{12}}}\oplus (g\tau)^*(\EE\otimes \WW)_{\mid_{\Delta_{12}}}\bigr)\otimes \fa_{12}   \Bigr)^{\sym_{12}}\cong g^*(\EE\otimes \WW)_{\mid_{\Delta_{12}}}\,; 
\end{equation}
see e.g.\ \cite[Remark 2.3]{krug_remarks_2018}.
Using the identification 
\[
X^{n-1}=X\times X^{n-2}\xrightarrow \cong \Delta_{12}\quad,\quad (x;x_3,\dots, x_n)\mapsto (x,x,x_3,\dots, x_n)\,,
\]
the right side of \eqref{eq:twosummandsinva} is of the form 
\[
g^*(\EE\otimes \WW)_{\mid_{\Delta_{12}}}\cong \bigl((E_i\otimes E_j)\boxtimes \cR\bigr)\otimes \WW 
\]
where $\WW$ just stands for the underlying $w$-dimensional vector space of the representation $\WW$, and $\cR$ is some permutation of the box-factors of
\[
E_1^{\boxtimes \lambda_1}\boxtimes \dots \boxtimes E_i^{\boxtimes \lambda_i-1}\boxtimes\dots \boxtimes E_j^{\boxtimes \lambda_j-1}\boxtimes\dots \boxtimes E_k^{\boxtimes \lambda_k}\,.
\]
In particular, $\rk g^*(\EE\otimes \WW)_{\mid_{\Delta_{12}}}=sw$. For a fixed pair $1\leqslant i<j\leqslant k$, there are exactly $p_{\lambda(\bar i, \bar j)}$ cosets in $\sym_\lambda\backslash \sym_n$ whose representatives satisfy $g(1)\in I_i$, $g(2)\in I_j$; see \eqref{eq:pnumber}. Hence, summing up over all $1\leqslant i<j\leqslant k$ we get a direct summand of $(G_{\mid \Delta_{12}}\otimes \fa_{12})^{\sym_{12}}$ of rank 
\begin{equation}
sw \left(\sum_{1\leqslant i<j\leqslant k} p_{\lambda}(\bar i, \bar j)\right) 
\end{equation}
which is exatly the first summand of the asserted formula. To get a complementary direct summand of $(G_{\mid \Delta_{12}}\otimes \fa_{12})^{\sym_{12}}$, we consider $g\in \sym_n$ with $g(1),g(2)\in I_i$ for some $i=1,\dots, k$. Note that then necessarily $\lambda_i\ge 2$. In this case $g$ and $g\tau$ represent the same coset in $\sym_\lambda\backslash \sym_n$, hence the $\sym_{12}$-action on $G_{\mid \Delta_{12}}\otimes \fa_{12}$ restricts to an action on the direct summand $g^*(\EE\otimes \WW)_{\mid_{\Delta_{12}}}\otimes \fa_{12}$. As an $\sym_{12}$-equivariant sheaf, this direct summand is given (up to the twist by $\fa_{12}$ that we will incorporate later) by 
\begin{align*}
g^*(\EE\otimes \WW)_{\mid \Delta_{12}}&\cong \bigl(E_i^{\otimes 2}\otimes \Res_{\sym_{\lambda_i}}^{\sym_2} W_i\bigr)\boxtimes \bigl(\cS\otimes \WW(\bar i)\bigr)\\&\cong \bigl((E_i^{\otimes 2})^{\oplus \alpha_i}\oplus (E_i^{\otimes 2}\otimes \fa_{12})^{\oplus \beta_i} \bigr)\boxtimes \bigl(\cS\otimes \WW(\bar i)\bigr)\,. 
\end{align*}
Here, $\WW(\bar i)=\otimes_{j\neq i} W_j$ equipped with the trivial $\sym_{12}$-action. In other words, $\WW(\bar i)$ is just a vector space of dimension $\prod_{j\neq i} w_j=\frac{w}{w_j}$. Furthermore, $\cS$ is some permutation of 
the box-factors of
\[
E_1^{\boxtimes \lambda_1}\boxtimes \dots \boxtimes E_i^{\boxtimes \lambda_i-2}\boxtimes E_k^{\boxtimes \lambda_k}\,.
\]
Tensoring with $\fa_{12}$ and then taking $\sym_{12}$-invariants gives
\[
\bigl(g^*(\EE\otimes \WW)_{\mid \Delta_{12}}\otimes \fa_{12} \bigr)^{\sym_{12}}\cong \bigl((\wedge^2 E_i)^{\oplus \alpha_i}\oplus (S^2E_i)^{\oplus \beta_i} \bigr)\boxtimes \bigl(\cS\otimes \WW(\bar i)\bigr) 
\]
which has rank $\bigr(\alpha_i\binom{r_i}2 + \beta_i\binom{r_i+1}2   \bigl)\frac{sw}{r_i^2w_i}$. For a given $i$, there are $p_\lambda(\bar i, \bar i)$ cosets in $\sym_\lambda\backslash \sym_n$ whose representatives satisfy $g(1)=g(2)\in I_i$; see \eqref{eq:pnumber}. Taking the sum over all $i$ with $\lambda_i\geqslant 2$ gives exactly the second summand of the asserted formula, namley
\[
\sum_{i \text{ with }\lambda_i\geqslant 2} \frac{p_{\lambda(\bar i, \bar i)}sw}{r_i^2 w_i}\Bigr(\alpha_i\binom{r_i}2 + \beta_i\binom{r_i+1}2\Bigl)\,.\qedhere
\]
\end{proof}

\begin{theorem}\label{thm:c1} We have
\[
c_1\bigl(F_\lambda^\WW(E_1,\dots ,E_k)\bigr) = \bigl(B_\lambda^\WW(E_1,\dots, E_k)\bigr)_{S^{[n]}} - R_\lambda^\WW(E_1,\dots, E_k)\cdot \delta
\]    
where 
\begin{align*}
B_\lambda^\WW(E_1,\dots, E_k)&= sw\bigl(\sum_{i=1}^k\frac{p_{\lambda(\bar i)}}{r_i} c_1(E_i)\bigl) \,\in \Pic(S)\,,\\
R_\lambda^\WW(E_1,\dots, E_k)&=sw \left(\sum_{1\leqslant i<j\leqslant k} p_{\lambda}(\bar i, \bar j)  + \sum_{i \text{ with }\lambda_i\geqslant 2} \frac{p_\lambda(\bar i, \bar i)}{r_i^2 w_i}\Bigr(\alpha_i\binom{r_i}2 + \beta_i\binom{r_i+1}2   \Bigl) \right)\in \ZZ\,.
\end{align*}
\end{theorem}

\begin{proof}
Since $F_\lambda^\WW(E_1,\dots ,E_k)=\Psi\bigl(G_\lambda^\WW(E_1,\dots ,E_k) \bigr)$, 
this is a combination of \autoref{prop:c1gen}, \autoref{lem:c1G}, and \autoref{lem:rkG}.     
\end{proof}

\subsection{Examples}

To illustrate how the formula of \autoref{thm:c1} works, we reprove some formulas from the literature as special cases.

The formula for the first Chern class looks much simpler in the special case that all the $W_i$, and hence also $\WW$, are trivial. Then $w_1=\dots=w_k=w=1$, $\alpha_1=\dots=\alpha_k=1$, $\beta_1=\dots=\beta_k=0$. Hence, \autoref{thm:c1} gives
\begin{equation}\label{eq:c1w1}
\begin{aligned}
&c_1\bigl(F_\lambda^{\mathbf 1}(E_1,\dots ,E_k)\bigr)\\ =&s\Bigl( \bigl(\sum_{i=1}^k\frac{p_{\lambda(\bar i)}}{r_i} c_1(E_i)\bigr)_{S^{[n]}} - \bigl(\sum_{1\le i<j\le k} p_{\lambda}(\bar i, \bar j)  + \sum_{i \text{ with }\lambda_i\ge 2} \frac{p_\lambda(\bar i, \bar i)}{r_i^2 }\binom{r_i}2 \bigr)\cdot \delta\Bigr)
\end{aligned}
\end{equation}
We get back the formula for the first Chern class of (non-generalised) tautological bundles, by considering $\lambda=(n-1,1)$, $E_1=\reg_S$, and $E_2=E$, which gives $F_{(n-1,1)}^{\mathbf 1}(\reg_S, E)\cong E^{[n]}$; compare \eqref{eq:getbacktaut}. In this case, we have $s=\rk(E)$. Furthermore, as $c_1(E_1)=c_1(\reg_S)=0$, only the $i=2$ part of the first sum in 
\eqref{eq:c1w1} contributes. Since $\lambda(\bar 2)=(n-1)$, hence $p_{\lambda(\bar 2)}=1$, we get $s \bigl(\sum_{i=1}^k\frac{p_{\lambda(\bar i)}}{r_i} c_1(E_i)\bigr)_{S^{[n]}}=\bigl(c_1(E)\bigr)_{S^{[n]}}$. For the $\delta$-part of \eqref{eq:c1w1}, we note that $\lambda(\bar 1,\bar 2)=(n-2)$, hence $p_{\lambda(\bar 1,\bar 2)}=1$, so the value of the first sum is 1. The second sum is zero, since $\lambda_i\ge 2$ only for $i=1$, but $r_1=1$, hence $\binom{r_1}2=0$. In summary, 
\[
c_1(E^{[n]})=\bigl(c_1(E)\bigr)_{S^{[n]}}-\rk(E)\cdot \delta\,,
\]
which agrees with \cite[Lemma 1.5]{wandel_tautological_2016}.

As a second special case, let us consider $\lambda=(1,\dots,1)$. We have $p_{\lambda(\bar i)}=(n-1)!$ for all $1\leqslant i\leqslant k$ and $p_{\lambda(\bar i,\bar j)}=(n-2)!$ for all $1\leqslant i< j\leqslant k$. As there are no $\lambda_i\geqslant 2$, the formula \eqref{eq:c1w1} specializes to 
\[
c_1\bigl(F_{(1,\dots,1)}^{\mathbf 1}(E_1,\dots ,E_n)\bigr) =(n-1)! s\Bigl( \bigl(\sum_{i=1}^k\frac{1}{r_i} c_1(E_i)\bigr)_{S^{[n]}} - \frac n2 \delta\Bigr)\,,
\]    
which agrees with the formula in \cite[Remark 2.10]{ogrady_many}.

\subsection{Repackaging the Formulas}

Still in the case that $\WW=\mathbf 1$, we can state our formula in a more compact way by summarizing the first Chern classes of $F_{\lambda}^{\mathbf 1}$ for all partitions $\lambda$ of $n$ in something similar to a counting function.
For this, the key observation is that, given locally free sheaves $E_1,\dots, E_n$ on $S$, all the $G_\lambda^{\mathbf 1}(E_1,\dots, E_k)$ occur as direct summands of 
$\cE^{\boxtimes n}$ where $\cE=E_1\oplus \dots \oplus E_n$. Hence, all the $F_\lambda^{\mathbf 1}(E_1,\dots, E_k)$
occur as direct summands of $\Psi(\cE^{\boxtimes n})$. The first Chern class of the latter is easy to compute using \autoref{prop:c1gen}, namely
\begin{equation}\label{eq:c1cE}
c_1\bigl( \Psi(\cE^{\boxtimes n})\bigr)=r^{n-1}\bigl( c_1(\cE) \bigr)_{S^{[n]}}-r^{n-2}\binom r2 \cdot \delta
\end{equation}
where $r=\rk(\cE)=r_1+\dots+r_n$. We can enrich \eqref{eq:c1cE} by introducing formal variables $t_1,\dots, t_n$ in order to keep track of the direct summands $E_i$ of $\cE$. More precisely, we consider 
\[
\cE_t:=E_1t_1+\dots+ E_n t_n\in K(S)[t_1,\dots, t_n]\,,\quad r_t:=r_1t_1+\dots +r_nt_n\in \CC[t_1,\dots, t_n]\,.
\]
Then, in the $\CC[t_1,\dots , t_n]$-module $\Pic(S^{[n]})[t_1,\dots ,t_n]$, we have the equality
\begin{equation}\label{eq:c1cEt}
c_1\bigl( \Psi(\cE_t^{\boxtimes n})\bigr)=r_t^{n-1}\bigl( c_1(\cE_t) \bigr)_{S^{[n]}}-r_t^{n-2}\binom{r_t} 2 \cdot \delta
\end{equation}
where the binomial coefficient is interpreted as $\binom{r_t} 2=\frac{r_t(r_t-1)}2$. We get back the formulas for the first Chern classes of the individual $F_\lambda^{\mathbf 1}(E_1,\dots, E_k)$ since $c_1\bigl(F_{(\lambda_1,\dots, \lambda_k)}^{\mathbf 1}(E_1,\dots E_k)\bigr)$ is the coefficient of $t_1^{\lambda_1}\cdots t_k^{\lambda_k}$ in \eqref{eq:c1cEt}. 

We can do something very similar if we replace the trivial by the sign representation. Let $\fa_n$ be the sign representation of $\sym_n$, and $\fa_\lambda=\Res_{\sym_n}^{\sym_\lambda} \fa_n\cong \fa_{\lambda_1}\otimes\dots\otimes \fa_{\lambda_k}$. Then, all $F_\lambda^{\fa_\lambda}(E_1,\dots, E_k)$ for varying partitions $\lambda$ of $n$ are direct summands of $\Psi\bigl(\cE^{\boxtimes n}\boxtimes \fa_n \bigr)$. The first Chern classes of an individual $F_\lambda^{\fa_\lambda}(E_1,\dots, E_k)$ can be recovered as the coefficient of $t_1^{\lambda_1}\cdots t_k^{\lambda_k}$ in
\[
c_1\bigl( \Psi(\cE_t^{\boxtimes n}\otimes \fa_n)\bigr)=r_t^{n-1}\bigl( c_1(\cE_t) \bigr)_{S^{[n]}}-r_t^{n-2}\binom{r_t+1} 2 \cdot \delta\,.
\]
On may wonder whether the formulas for the first Chern classes of $F_\lambda^\WW(E_1,\dots, E_k)$ for all partitions $\lambda$ and \emph{all} representation $\WW\in \irr(\sym_\lambda)$ can be summarised in a similarly elegant way. At least, all the $G_\lambda^\WW(E_1,\dots, E_k)$ are contained in $\cE^{\boxtimes n}\otimes \CC\langle \sym_n\rangle$ where $\CC\langle \sym_n\rangle$ is the regular $\sym_n$-representation. Furthermore, \autoref{prop:c1gen} gives
\begin{equation}\label{eq:c1reg}
c_1\bigl(\Psi(\cE^{\boxtimes n}\otimes \CC\langle \sym_n\rangle)\bigr)=n!r^{n-1}\bigl( c_1(\cE) \bigr)_{S^{[n]}}-\frac{n!}2 r^{n} \cdot \delta\,.
\end{equation}
The question remains whether \eqref{eq:c1reg} can be enriched in a not too complicated way, such that one can recover the formula for the individual $c_1\bigl(F_\lambda^\WW(E_1,\dots, E_k)\bigr)$ from it.

\bibliography{exceptional}
\bibliographystyle{alpha}
\end{document}